\documentclass[11pt, article]{elsarticle}
\usepackage{lineno}
\modulolinenumbers[1]

\usepackage{geometry}
\usepackage{amsfonts}
\usepackage{amsmath}
\usepackage{amssymb}
\usepackage[mathscr]{eucal}
\usepackage{blkarray}
\usepackage{epsfig}
\usepackage{fullpage}
\usepackage{times}
\usepackage{multirow}
\usepackage{float}
\usepackage[usenames]{color}
\usepackage[dvipsnames]{xcolor}
\usepackage{hyperref} 

\usepackage{xcolor}
\usepackage{framed}
\definecolor{shadecolor}{rgb}{1,0,0}
\newcommand{\fby}{\fcolorbox{blue}{yellow}}

\allowdisplaybreaks

\def\eod{\vrule height 6pt width 5pt depth 0pt}
\newenvironment{proof}{\noindent {\bf Proof:} \hspace{.2em}}
{\hspace*{\fill}{\eod}}

\newcommand{\rk}{\mathrm{rank}}

\newcommand{\VV}{\mathcal{V}_2}
\newcommand{\Max}{\mathrm{Max4PC}}
\newcommand{\Min}{\mathrm{Min4PC}}
\newcommand{\LG}{\mathrm{LG}}

\newcommand{\bone}{\mathbf{1}}

\newcommand{\BB}{\mathfrak{B}}

\newcommand{\DD}{\mathfrak{D}}

\newcommand{\RR}{\mathbb{R}}

\newcommand{\ZZ}{ \mathbb{Z}}
\newcommand{\Tr}{ \mathrm{Trace}}
\newcommand{\charpoly}{ \mathrm{CharPoly}}
\newcommand{\Inertia}{ \mathrm{Inertia}}

\newtheorem{theorem}{Theorem}

\newtheorem{corollary}[theorem]{Corollary}
\newtheorem{conjecture}[theorem]{Conjecture}

\newtheorem{remark}[theorem]{Remark}

\newtheorem{definition}[theorem]{Definition}
\newtheorem{lemma}[theorem]{Lemma}

\newcommand{\comment}[1]{}
\newcommand{\deleted}[1]{}
\newcommand{\red}[1]{\textcolor{red}{#1}}

\newcommand{\blue}[1]{\textcolor{blue}{#1}}

\begin{document}
	
	\begin{frontmatter}		
		
		% \title{Unimodality and Peak Location in the  Characteristic Polynomials of Tree-Based Matrices}
		\title{Unimodality and peak location of the  characteristic polynomials of two distance matrices of trees} 
		
% \title{Further extensions of the Graham and Lov{\'a}sz conjecture on the 
% characteristic polynomial of distance matrices}
		\author{Rakesh Jana%\corref{cor1}
		}
		\ead{rjana@math.iitb.ac.in}
		\author{Iswar Mahato}
		\ead{iswar@math.iitb.ac.in}
		\author{Sivaramakrishnan Sivasubramanian}
		\ead{krishnan@math.iitb.ac.in}
		\address{Department of Mathematics, Indian Institute of Technology Bombay, Mumbai 400076, India}
	
		\begin{abstract} 
%A solution to a conjecture of Graham and Lov{\'a}sz about the
Unimodality of the normalized coefficients of the  characteristic polynomial of distance matrices of trees are known and bounds on the location of its peak (the largest coefficient) are also known. Recently, an extension of these results to distance matrices of block graphs was given. In this work, we extend these results to two additional distance-type matrices associated with trees: the Min-4PC matrix and the 2-Steiner distance matrix. We show that the sequences of coefficients of the characteristic polynomials of these matrices are both unimodal and log-concave. Moreover, we find the peak location for the coefficients of the characteristic polynomials of the Min-4PC matrix of any tree on $n$ vertices. Further, we show that the Min-4PC matrix of any tree on $n$ vertices is isometrically embeddable in $\RR^{n-1}$ equipped with the $\ell_1$ norm. 
% two more matrices associated to trees: the Min-4PC matrix and the 2-Steiner  distance matrix.
		\end{abstract}
		
	\end{frontmatter}
	
	%	\linenumbers
	
%%%%%%%%%%%%%%%%%%%%%%%%%%%%%%%%%%%%	
	\section{Introduction}
	\label{sec:intro}
Let $A = a_0,a_1,\cdots,a_m$ be a sequence of real numbers. The sequence $A$ is called  
	\emph{unimodal} if there exists an index $k$ with $1 \leq k \leq m-1$ 
	such that $a_{j-1}\leq a_j$ 
	when $j\leq k$ and $a_j\geq a_{j+1}$ when $j\geq k$.  The  sequence $A$
	%$a_0,a_1,\cdots,a_n$ 
is called \emph{log-concave} if 
	$a_i^2\geq a_{i-1}a_{i+1}$ when $1 \leq i \leq m-1$. Log-concavity and unimodality are significant properties with application across various areas; for example, algebra (see \cite{branden-unimodality_log_concavity} by Br{\"a}nd{\'e}n and \cite{stanley1989log} by Stanley), probability theory  (see \cite{prekopa1971logarithmic} by Prekopa), combinatorics and geometry \cite{stanley1989log}. These applications emphasize the importance of understanding and identifying log-concave sequences in different mathematical contexts. For the distance matrix of a tree, Graham and Lov{\'a}sz in \cite[page 83]{graham-lovasz} conjectured unimodality of the normalized coefficients of the characteristic polynomial of the distance matrix of trees and also conjectured the location of the peak(s). The unimodality part was proved by Aalipour et. al. in 
	\cite{aalipour-hogben-etal-conjecture-graham-lovasz} and the peak location was disproved by Collins in \cite{collins-disproof-graham-lovasz-peak}.
	
	Two points are noteworthy.  Firstly, 
	for a square matrix $M$, the definition of its characteristic 
	polynomial used in all earlier papers is $\chi_M(x) = \det(M - xI)$ 
	and this does not always make 
	$\chi_M(x)$ monic.  We thus change the definition slightly
	and define 
	\begin{equation}
		\label{eqn:charpoly_monic}
		\charpoly_M(x) = \det(xI - M).  
	\end{equation}
	Doing this small change helps us get rid of the need to multiply 
	its coefficients by a power of $(-1)$ with the exponent depending on $n$.  
	%The reason for doing so in their conjecture with the earlier
	%definition is mentioned by 
	%Graham and Lovasz in \cite[Equation 44]{graham-lovasz}.  
	Hence, for the rest of this paper, for square matrices $M$, we define
	$\charpoly_M(x)$ using \eqref{eqn:charpoly_monic} and make the 
	needed small changes to results before  quoting them 
	from the literature.
	
	Secondly, normalizing the coefficients of the  characteristic 
	polynomial is not important for unimodality as we get similar 
	results by scaling the coefficients $c_k$ of the characteristic polynomial 
	by $\alpha^k$, where $\alpha$ is a positive real number.  This point is also mentioned by both
	Abiad et. al 
	in \cite[Section 4]{Abiad-Brimkov-Hayat-Khramova-Koolen-dist-char-poly} 
	and by Aalipour et. al in
	\cite[Observation 1.3]{aalipour-hogben-etal-conjecture-graham-lovasz}.  
	However, to determine the peak location of a unimodal sequence (or for
	bounds on the peak location), it is important to know whether we 
	take the coefficients of the characteristic polynomial or its 
	normalized version.  
	In this paper, for results on the peak location, we consider the sequence
	of coefficients of $\charpoly_M(x)$ without any normalization.  
	Abiad et. al in 
	\cite{Abiad-Brimkov-Hayat-Khramova-Koolen-dist-char-poly} also 
	give their results for the coefficients of the un-normalized characteristic polynomial.
	%though we mention appropriate 
	%changes that need to be made for its normalized version.

	Let $T$ be a tree with $n$ vertices and let $D$ be its distance
	matrix.  Let $\charpoly_D(x) = \det (xI - D) = \sum_{k=0}^n c_k x^k$
	be $D$'s characteristic polynomial.  By definition, as
	$\charpoly_D(x)$ is a monic polynomial, $c_n = 1$.  
	We further have $c_{n-1} = 0$ as $c_{n-1} = \Tr(D)$ which is 
	zero (as all diagonal entries of distance matrices are zero). In \cite{aalipour-hogben-etal-conjecture-graham-lovasz}, Aalipour et. al proved the following result. 
	
	\begin{theorem}[Aalipour et. al]
		\label{thm:unimodal-log-conc}
		With the notation above, let $\displaystyle d_k = \frac{-1}{2^{n-k-2}} 
		c_k$ be the normalized coefficients of $\charpoly_D(x)$.  Then, the 
		sequence $d_k$ as $k$ varies from $0$ to $n-2$ is unimodal and log concave.
	\end{theorem}
	
	The proof of Theorem \ref{thm:unimodal-log-conc} uses real 
	rootedness of $\charpoly_D(x)$ to show log 
	concavity. For unimodality, they need the following result 
of Edelberg, Garey and Graham (see \cite[Theorem 2.3]{edelberg-garey-graham}) 
which states that when $0 \leq k \leq n-2$, 
	$c_k$ is negative (and hence $d_k$ is positive).
	
\begin{theorem}[Edelberg, Garey and Graham]
\label{thm:edelberg-garey-graham}
With $T$ and $D$ as above, let $\displaystyle \charpoly_D(x) = \sum_{k=0}^n c_k x^k.$  
Then, for $0 \leq k \leq n-2$, we have $c_k < 0$.
\end{theorem}

	An extension of these results to distance matrices
	of block graphs was obtained by 
	Abiad et al. 
	in \cite{Abiad-Brimkov-Hayat-Khramova-Koolen-dist-char-poly}.  
	The authors showed unimodality results for 
	 coefficients in the characteristic polynomial of distance 
	matrices  of block graphs along similar lines and gave bounds on 
	the peak location for some block graphs.

	In this paper, we extend such results to two other matrices.  Both
	matrices are defined for trees $T$.  The first, $\Min_T$ is very similar to the
	distance matrix $D_T$ of trees $T$ while the second is 2-Steiner
	distance matrix $\DD_2(T)$ of trees $T$.   This second matrix does 
	not have diagonal elements that are zero but our proof goes through nonetheless.
	Both of these are $\binom{n}{2} \times \binom{n}{2}$ matrices but are 
	not full rank matrices (see \cite{aliazimi-siva-steiner-2-dist, bapat-siva-snf-4PC}), so our results are for the restriction of these
	matrices to a basis for their respective row spaces.  
	
	Let $T$ be a tree on $n$ vertices and let $\VV$ be the set of 2-element 
	subsets of the vertices of $T$.  Clearly $|\VV| = \binom{n}{2}$.  Define the following 
	$\binom{n}{2} \times \binom{n}{2}$ matrices whose rows and columns are indexed by
	elements of $\VV$.  
	Let $d_{i,j}$ be the distance between $i$ and $j$ in $T$.  For four
	vertices $i,j,k$ and $l$ from the vertex set of $T$, define the 
	(multi)set 
	$$S_{i,j,k,l} = \{ d_{i,l} + d_{j,k}, d_{i,k} + d_{j,l}, 
	 d_{i,j} + d_{k,l}\}.$$  
	Tree distances are special and Buneman 
	in \cite{buneman-4PC} showed that for all choices $i,j,k,l$ of four 
	vertices, among the three terms in $S_{i,j,k,l}$, the second maximum value 
	equals the maximum value.  This inspired the definition of 
	the following $\binom{n}{2} \times \binom{n}{2}$ matrices.  
	
	Define the $\Min_T$ matrix as follows.  For $\{i,j\}, \{k,l\} \in \VV$,
	the entry of $\Min_T$ corresponding to the row $\{i,j\}$ and the column $\{k,l\}$ 
	is the minimum entry of $S_{i,j,k,l}$.  One can also define 
	the $\Max_T$ matrix by changing the word ``minimum" in the 
	previous sentence to ``maximum".    For a tree $T$, Azimi and Sivasubramanian 
	in \cite{aliazimi-siva-steiner-2-dist} defined $\DD_2(T)$, 
	the 2-Steiner distance matrix of a tree $T$ as follows. For $\{i,j\}, \{k,l\} \in \VV$, the entry of
	$\DD_2(T)$ corresponding to the row $\{i,j\}$ and column $\{k,l\}$ 
is the minimum number of edges among all connected subtrees of $T$ whose vertex set contains 
the four vertices $i,j,k$ and $l$.
	Azimi and Sivasubramanian showed that $\DD_2(T)$ is the average of the $\Max_T$ and $\Min_T$ matrix, that is, $\DD_2(T) = \frac{1}{2} \Big(\Max_T + \Min_T \Big)$.

Bapat and Sivasubramanian in
\cite{bapat-siva-snf-4PC} studied the $\Min_T$ matrix and showed 
results on its rank and  its invariant factors.  
Consider a tree $T = (V,E)$ on $n$ vertices.  Let $j,k \in V$ with
$j\not= k$ be two vertices and let $f = \{j,k\} 
\not\in E$ be a non edge of $T$ with $d_{j,k} = d > 1$ 
(that is, the distance in $T$ between $j$ and $k$ is $d$).
Bapat and Sivasubramanian in \cite{bapat-siva-snf-4PC} proved that 
$\rk(\Min_T) = n$ and  
the set $B = E \cup \{f\}$ forms basis of $\Min_T$'s row space.  
Our first result is the following about $\Min_T[B,B]$, the 
submatrix of $\Min_T$ restricted to both the rows and columns in $B$.

\begin{theorem}
\label{thm:unimodal-log-conc-min4pc} 
With the notation above, let $N = \Min_T[B,B]$ and  $\charpoly_N(x) =  \sum_{k=0}^n a_kx^k$. Then, the sequence $|a_k|$  as $k$ varies from $0$ to $n-2$ is unimodal and log-concave. If $|a_t| = \max_{0 \leq k \leq n-2} |a_k|$  is the largest coefficient in absolute value, then  $\lfloor \frac{n-2}{3} \rfloor \leq t \leq \lceil \frac{n+1}{3}\rceil$.
\end{theorem}

	When $T$ is a tree of order 
	$n$ with $p$ leaves and $B \in \BB$ is a basis of $\DD_2(T)$'s row space, the authors in 
	\cite[Theorem 18]{aliazimi-siva-steiner-2-dist} proved that 
	$\DD_2(T)[B,B]$ has $2n-p-2$ negative eigenvalues and one positive eigenvalue. In this paper, we show that when $B$ is a basis of $\Min_T$'s row space, we get an analogous statement 
	for the matrix $\Min_T[B,B]$.   This is proved in two ways
	with our first proof being Theorem \ref{thm:inertia-min4pc} proved in Section \ref{sec:min4pc-results}. Our second proof is more general and is of independent interest 
	as it gives some corollaries about hypermetricity and negative-type
	metric spaces which we do not get from our first proof.
 	In Section \ref{sec:ell1}, we 
	give an isometric embedding of $T$'s $\binom{n}{2} \times 
	\binom{n}{2}$  $\Min_T$ matrix into $\RR^{n-1}$ equipped with 
	the $\ell_1$ norm.  We prove the following result.
	
	\begin{theorem}
		\label{thm:l1_embedding}.  
		Let $T$ be a tree having $n$ vertices.
		%and let $B = E \cup \{f\}$
		%where $f = \{i,j\}$ and $f \not\in E$.  Let $\Min_T[B,B]$ be 
		%the matrix $\Min_T$ restricted to the rows and columns in $B$. 
		Then, $\Min_T$ is isometrically $\ell_1$-embeddable in
		$\RR^{n-1}$.
	\end{theorem}
	
	Our proof is surprisingly easy and appears in Section \ref{sec:ell1}. For all trees $T$, it follows from the theory of isometrically $\ell_1$-embeddable finite metric spaces (see Deza and Laurent 
	\cite[Chapter 19]{deza-laurent-geometry-cuts-metrics}) that the $\Min_T$ matrix has exactly one positive eigenvalue. By standard interlacing arguments, restricting $\Min_T$ to elements 
	from a basis $B$, if $\Min_T[B,B]$ is a full rank matrix having
	rank $r$,  one infers that $\Min_T[B,B]$ has $r-1$ negative eigenvalues
and 1 positive eigenvalue.

	A distance matrix $D = (d_{i,j})_{1 \leq i,j \leq n}$ is said to be 
	a {\sl hypermetric} if 
	\begin{equation}
		\label{eqn:ineq-for-metr}
		\sum_{1 \leq i < j \leq n} x_ix_jd_{i,j} \leq 0 
	\end{equation}
	for all $x \in \ZZ^n$ with $\sum_{i=1}^n x_i = 1$ ($x_i$ here is the $i$-th
	component of $x$).  If inequality 
	\eqref{eqn:ineq-for-metr} holds for  all $x \in \ZZ^n$ 
	with $\sum_{i=1}^n x_i = 0$, 
	then $D$ is said to be a {\sl negative type metric.}  It is known (see
	\cite[Chapter 6]{deza-laurent-geometry-cuts-metrics}) that if a distance 
	matrix $D$ is isometrically embeddable in an $\ell_1$ space, 
	then it is both {\sl hypermetric} and {\sl of negative type.}
	For any tree $T$, though the matrix $\Min_T$ satisfies the 
	triangle inequality, proving this takes some work.  Remark
	\ref{rem:triangle-ineq} shows that this can be 
	obtained as a simple consequence of our isometric embedding.

	Azimi and Sivasubramanian in \cite{aliazimi-siva-steiner-2-dist} 
	considered the matrix $\DD_2(T)$. 
	Note that the diagonal entry of $\DD_2(T)$ corresponding to
	 the row and column indexed by $\{i,j\} \in \VV$ equals $d_{i,j}$, 
	 which is the tree distance between $i$ and $j$.  
	 Hence, $\DD_2(T)$
	does not have zero entries in its diagonal (indeed all its 
	main diagonal entries are positive).  For a tree $T$ of order $n$ with $p$ leaves
	Azimi and Sivasubramanian
	showed that $\rk(\DD_2(T)) = 2n-p-1$, gave a class $\BB$ of bases for 
	its row space and obtained the determinant of 
	$\DD_2(T)[B,B]$,  the restriction of $\DD_2(T)$ to the entries in rows and columns from $B \in \BB$. In this article, we obtain the following result about $\DD_2(T)[B,B]$.

\begin{theorem}
		\label{thm:unimodal-log-conc-2steiner-dist} 
		Let $T$ be a tree on $n$ vertices and let $T$ have $p$ leaves.   With the 
		notation above, for any $B \in \BB$, consider $P = \DD_2(T)[B,B]$ 
		and let $\charpoly_P(x) = \sum_{k=0}^{2n-p-1} a_kx^k$.
		Then, the sequence $|a_k|$ as $k$ varies from $0$ to $2n-p-2$
		is unimodal and log-concave.   
	\end{theorem}
	
	%We consider two singular matrices in this work.   
	\comment{
		\blue{
			Though we start with singular matrices, by restricting attention 
			to entries from a basis, we consider two full rank $r \times r$ 
			matrices $M$.  Further, all our matrices will have 
			exactly one positive eigenvalue and $(r-1)$ negative eigenvalues.
			\red{Check that the result quoted below is on matrices with all
				diagonal entries being 0.}
			\blue{We give a mild generalization of 
				\cite[Lemma 4.1]{Abiad-Brimkov-Hayat-Khramova-Koolen-dist-char-poly}
				on the $\charpoly_M(x)$ where $M$ 
				has exactly one positive eigenvalue and has non negative 
				entries along its main diagonal.}   This is given as Theorem 
			\ref{coeff-same-sign} in Section \ref{sec:prelims}.  
			%Such a result 
			%is known for distance matrices and typically, all diagonal 
			%entries of distance matrices are zero.  Hence, 
			We need this version for the matrix $\DD_2(T)[\BB,\BB]$,
			as it has non-zero elements on its diagonals.
	}}

	A uniform proof giving bounds on the peak location of the coefficients 
of $\charpoly_{\DD_2(T)[B,B]}(x)$ for all trees $T$ seems hard.  So, we consider three 
	special cases, the star tree, the bi-star tree $S_{1,n-3}$ and the path tree 
and obtain bounds on $|a_t| = \max_{0 \leq k \leq 2n-p-3} |a_k|$,
	the largest coefficient in absolute value in their 
	respective characteristic polynomials.  For the star and 
	the bi-star our bounds are tight and are given as
	Theorem  \ref{peak-steiner-star} and Theorem 
	\ref{peak-location-steiner-bistar} in Subsections 
	\ref{subsec:peak-star} and \ref{subsec-2-steiner-bistar}, respectively.  
	For the path tree, we give an upper bound on the 
	peak location as Theorem \ref{th: peak-steiner-path} 
	in Subsection \ref{subsec:peak-locn-path}, and we 
conjecture the value of the peak location.

%%%%%%%%%%%%%%%%%%%%%%%%%%%%%%%%%%%%%%%%%%%%%%%%%%%%%%%%%%%%%%%
\section{Unimodality and log-concavity}
\label{sec:prelims}
	For unimodality, we will need the idea of {\sl real rootedness} of polynomials 
	with real coefficients.  The following result 
	\cite[Lemma 7.1]{branden-unimodality_log_concavity} is known.

	\begin{lemma}
		\label{unimodal-condition}
		Let $p(x)=\sum_{k=0}^n a_kx^k$ be a real-rooted polynomial with real 
		coefficients.  
		
\begin{enumerate}
\item Then its coefficient sequence $a_0,a_1,\hdots,a_n$ is 
log-concave.
\item If a sequence $a_0,a_1,\hdots,a_n$ is both positive and 
	log-concave, then it is unimodal.
\end{enumerate}    
\end{lemma} 
	
	For any real and symmetric matrix $M$, by the Spectral Theorem, 
	$\charpoly_M(x)$ is real rooted and so the first part of 
	Lemma \ref{unimodal-condition} is  trivially satisfied.
	When all eigenvalues of $M$ are negative, it is easy to see that all
	coefficients of $\charpoly_M(x)$ are positive.  
	
	When $M = (m_{i,j})_{1 \leq i,j \leq n}$ is an $n \times n$ real, symmetric 
	matrix with $m_{i,i} = 0$ for $1 \leq i \leq n$, and if $M$ has 
	exactly one positive eigenvalue then, the 
	proof of Theorem \ref{thm:edelberg-garey-graham} can be extended
	to show that almost all the coefficients of $\charpoly_M(x)$ 
	are negative.  This is the main point of 
	\cite[Lemma 4.1]{Abiad-Brimkov-Hayat-Khramova-Koolen-dist-char-poly}.

	Below, we mildly generalize this to include real, symmetric 
	matrices which have a non negative trace.  Recall 
	the inertia of a real symmetric matrix $M$ is the triple 
	$\Inertia(M) = \big(n_{+}(M),n_{-}(M),n_{0}(M)\big)$.  Here, $n_{+}(M),$ 
	$n_{-}(M)$ and $n_{0}(M)$ denote the number of positive, negative and zero 
	eigenvalues of $M$ respectively.

	\begin{theorem}
\label{coeff-same-sign}
Consider a  real, symmetric matrix $M$ of order $n$ with $\Tr(M)\geq 0$ and $\charpoly_M(x)=\sum_{k=0}^n a_kx^k$. 
%Suppose that the characteristic polynomial $\charpoly_M(x)=\sum_{k=0}^n 
%c_kx^k$ of $M$ is real-rooted and 
Let $\Inertia(M)=(1,r-1,n-r)$, with $2\leq r\leq n$. If 
$\Tr(M)=0$, then $a_k<0$ when $n-r\leq k\leq n-2$ and 
if $\Tr(M)> 0$, then $a_k<0$ when $n-r\leq k\leq n-1$.  
\end{theorem}
	\begin{proof}
		%Since $M$ is real and symmetric, let its real eigenvalues
		%be $\lambda_1, \lambda_2, \ldots, \lambda_n$.   
		Let the non zero eigenvalues of $M$ be $\lambda_1, -\lambda_2, -\lambda_3,
		\ldots, -\lambda_r$ 
		and let the eigenvalue $0$ occur with multiplicity $n-r$.  Here, we 
		assume that $\lambda_i > 0$ when $1 \leq i \leq r$ and that the 
		$\lambda_i$'s need not be distinct.  Define $g_0 = 1$ and
		when $k \geq 1$, define 
		$g_k$ to be the sum of all $k$-fold products of $\lambda_2, \ldots, 
		\lambda_n$.  Clearly, $g_k > 0$ when $1 \leq k \leq r$.  Further
		\begin{eqnarray}
			\charpoly_M(x) & = & x^{n-r}(x-\lambda_1)\prod_{i=2}^r(x+\lambda_i)  =  x^{n-r}(x - \lambda_1)\Bigg( \sum_{k=0}^{r-1} g_k x^{r-1-k}\Bigg) \nonumber \\
			& = & \Bigg( x^n + \sum_{k=1}^{r-1} \big( g_k - \lambda_1g_{k-1} \big) x^{n-k} -\lambda_1 g_{n-1} x^{n-r} \Bigg) \label{eqn:imp}
		\end{eqnarray}
		Since $\lambda_1=g_1+t$, $g_k - \lambda_1g_{k-1}=g_k - (g_1+t)g_{k-1}
		=(g_k-g_1g_{k-1})-tg_{k-1}<0 $ as we have $t\geq 0$ and 
		$ -\lambda_1 g_{n-1}= -(g_1+t) g_{n-1}<0$. 
		Moreover, $c_{n-1}=-\Tr(M)=-t$.  Hence, 
		when $n-r\leq k\leq n-1$ and  $t>0$, we have $a_k<0$.
		Likewise, when $n-r\leq k\leq n-2$ and $t=0$, we have 
		$a_k<0$, completing the proof.
	\end{proof}
	
	The following corollary of Theorem \ref{coeff-same-sign} can be drawn.
	
\begin{corollary}
\label{cor:dist-matrix-unim}
		Let $M$ be a real and symmetric matrix of order $n$ with $\charpoly_M(x)=\sum_{k=0}^n a_kx^k$ and $\Inertia(M)=(1,n-1,0)$. 
		\begin{enumerate}
			\item If $\Tr(M)=0$, then the sequence $|a_0|, |a_1|, \hdots,
			|a_{n-2}|$ 
of the  absolute values of its coefficients from $\charpoly_M(x)$ 
is log-concave and unimodal.
			\item If $\Tr(M)>0$, then the sequence $|a_0|, |a_1|, \hdots,
			|a_{n-2}|, |a_{n-1}|$ of the  absolute values of its coefficients 
from $\charpoly_M(x)$ is log-concave and unimodal.
		\end{enumerate}  
	\end{corollary}
	\begin{proof}
		Since $M$ is a real, symmetric matrix, 
	$\charpoly_M(x)$ is real-rooted and hence by Lemma \ref{unimodal-condition}, it 
follows that the sequence $a_0,a_1,\hdots,a_{n-2},a_{n-1}, a_n$ is log-concave.

1. By Theorem \ref{coeff-same-sign}, we get $a_k<0$ when $0\leq k\leq n-2$. Since 
all terms $a_0,a_1,\hdots,a_{n-2}$ are negative, the sequence comprising their absolute 
values $(|a_k|)_{k=0}^{n-2}$ is log-concave and positive.  By Lemma 
\ref{unimodal-condition},  $|a_0|, |a_1|, \hdots, |a_{n-2}|$ is unimodal.

2. By Theorem \ref{coeff-same-sign}, we have $a_k<0$ for $0\leq k\leq n-1$.  As 
all the terms $a_0,a_1,\hdots,a_{n-1}$ are negative, the sequence 
comprising their absolute values $(|a_k|)_{k=0}^{n-1}$ is log-concave and positive. 
By Lemma \ref{unimodal-condition}, $|a_0|, |a_1|, \hdots,
		|a_{n-1}|$ is unimodal.
		%The proof is clear from Lemma \ref{unimodal-condition} and Theorem \ref{coeff-same-sign}.
		%\blue{Perhaps give a little more detail.  For the third result, since $c_n=1$ and $c_{n-1}=\Tr(M)\geq 1$, unimodality follows.}
	\end{proof}

	% section 3
	
\section{The $\Min_T$ matrix of a tree $T$}
\label{sec:min4pc-results}
Let $T= \big(V,E\big)$ be a tree with $V=\{1,2,\hdots,n\}$.
Further, let $E=\{e_1,e_2,\hdots,e_{n-1}\}$. If $i,j \in V$ with $f=\{i,j\} 
	\not\in E$ be a non-edge of $T$ with $d_{i,j} = d > 1$, 
	then Bapat and Sivasubramanian in \cite{bapat-siva-snf-4PC} showed 
	that $B = E \cup \{f\}$ is a basis of $\Min_T$'s row space. 
	Consider the $n \times n$ matrix $N = \Min_T[B,B]$ obtained by 
	restricting the matrix $\Min_T$ to its rows and columns in $B$. 
We start this section with the following result.
	
	\begin{theorem}\label{thm:inertia-min4pc}
Let $N = \Min_T[B,B]$ be the matrix as described above. Then, $N$ has 
$(n-1)$ negative eigenvalues and  one positive eigenvalue.
	\end{theorem}
	\begin{proof}
		In our proof, we use the Schur complement formula for inertia.  The matrix $N$ 
		restricted to the 
		rows and columns indexed by $E$ is $K = 2(J-I)$ (see \cite[Lemma 3]{bapat-siva-snf-4PC}) 
		whose inverse is also presented in \cite[Lemma 4]{bapat-siva-snf-4PC}. Clearly, 
$K$ has $(n-2)$ negative eigenvalues and one positive eigenvalue.
		Further, let $x_f$ be an $(n-1)$-dimensional 
		column vector with its columns indexed by $e \in E$ with its 
		$e$-th component $x_f(e) = \Min_T(f,e)$.  Then, the Schur complement of $K$ in 
		$N$ equals $0 - x_f^tK^{-1}x_f$.  By 
		\cite[Corollary 7]{bapat-siva-snf-4PC}, this equals $p = -\frac{n-1}{2(n-2)}$.  Since 
		$\Inertia(N) = \Inertia(K) + \Inertia(p)$, we get that $N$ has only 
		one positive eigenvalue and $n-1$ negative eigenvalues.  
	\end{proof}

	%Theorem \ref{thm:inertia-min4pc} ensures that Corollary \ref{cor:dist-matrix-unim} can be used.  To find bounds on the peak location, one typically needs some coefficients of $\charpoly_M(x)$.  Below, using equitable partitions, we determine $\charpoly_M(x)$ explicitly.\blue{The book by Brouwer and Haemmers \cite[Chapter ???]{brouwers-haemers-spectra} is a good reference for this topic.}}

	To give our proof of Theorem \ref{thm:unimodal-log-conc-min4pc}, we 
	compute $\charpoly_N(x)$  
	%the characteristic polynomial of $N = \Min_T[B,B]$ 
	using equitable partitions.  We first recall the definition 
	of an equitable partition of a matrix $M$.
	
	\begin{definition}[Equitable Partition]
	Let $M$ be an $n \times n$ real, symmetric matrix and index the rows 
	and columns of $M$ by elements of the set $X$. Let
	$\Pi=\{X_1,X_2,\hdots,X_p\}$ be a partition of the set $X$ and let $M$ 
be partitioned according to $\Pi$ as
	\[M=\left( {\begin{array}{cccc}
			M_{11} & M_{12} &\hdots & M_{1p}\\
			M_{21} & M_{22} &\hdots & M_{2p}\\
			\vdots &\vdots & \ddots & \vdots\\
			M_{p1} & M_{p2}& \hdots &M_{pp}\\
	\end{array} } \right).\]
	Here, $M_{ij}$ denotes the block submatrix of $M$ induced  
	by the rows in $X_i$ and the columns in $X_j$. 
	If the row sum of each block $M_{ij}$ is a constant, 
	then the partition $\Pi$ is called an equitable partition.
	Let $q_{ij}$ denote
	the average row sum of $M_{ij}$.   The matrix $Q=(q_{ij})$ 
	is called the quotient matrix of $M$ with respect to $\Pi$. 
	\end{definition}
	
	Next, we state a well-known result (see \cite[Lemma 2.3.1]{brouwers-haemers-spectra})
	connecting the spectrum of a quotient matrix arising from an equitable partition
	to the spectrum of the original matrix.
	
	\begin{lemma}
	\label{lem:quo-spec}
	Let $Q$ be a quotient matrix of any real, symmetric, square matrix $M$ arising from 
	an equitable partition. Then, all eigenvalues of $Q$ are eigenvalues of $M$.
	\end{lemma}
	
	We next find the spectra of $\Min_T[B,B]$ for any tree $T$ of order $n$.
	
	\begin{theorem}\label{spectra-min4pc}
	For any tree $T$ on $n$ vertices, the eigenvalues of $\Min_T[B,B]$ are $-2$ with 
	multiplicity $n-3$, and the three roots of the cubic polynomial 
	$$g(x)=x^3-(2n-6)x^2-(nd^2-5d^2+2nd-2d+5n-9)x-2(d-1)^2(n-1).$$
	\end{theorem}
	\begin{proof}
	Let $f = \{i,j\}$ with $d_{i,j} = d$.  By relabelling, we can assume 
	that the edges $e_1, e_2, \ldots, e_d$ are on the $ij$-path in $T$.
	Let $E_1=\{e_1,e_2,\hdots,e_d\}$ and $E_2=\{e_{d+1},\hdots,e_{n-1}\}$. 
	Let $J$ denote a matrix all of whose entries are 1 (of appropriate
	dimension) and $I$ denoting the identity matrix (of appropriate
	dimension), and $\bone$ denote a column vector all of whose components
	are 1 (of appropriate dimension).  With these, $N=\Min_T[B,B]$ can be written as
	\[ N=
	\begin{blockarray}{cccc}
		& E_1 & E_2  & f\\
		\begin{block}{c(ccc)}
			E_1 & 2(J-I)  & 2J & (d-1) \bone \\
			E_2 & 2J & 2(J-I) & (d+1) \bone \\
			f & (d-1)\bone^t & (d+1)\bone^t &  0\\
		\end{block}
	\end{blockarray} ~.\]
	
	Let $e(i,j)$ denote the $n$-dimensional column vector that has its $i$-th component 
	$1$, its $j$-th component $-1$ and all other components as $0$. 
	If $S=\{e(j,j+1):1\leq j\leq d-1\}\cup \{e(j,j+1):d+1\leq j\leq n-2\}$, then for 
	any $\textbf{x}\in S$, we have $N\textbf{x}=-2\textbf{x}$. 
	Note that $|S|=n-3$, and that all vectors in $S$ are linearly independent. 
	Therefore, $-2$ is an eigenvalue of $N$ with multiplicity at least $n-3$.
	
	Recall that $E_1=\{e_1,e_2,\hdots,e_d\}$ and $E_2=\{e_{d+1},\hdots,e_{n-1}\}$. 
	Then it is 
	easy to check that $\Pi_1=E_1 \cup E_2\cup \{f\}$ is an equitable partition of $N$ 
	with quotient matrix 
	\[Q_{\Pi_1}=
	\left ( {\begin{array}{ccc}
			2(d-1) & 2(n-d-1) & d-1\\
			2d & 2(n-d-2) & d+1 \\
			d(d-1) & (d+1)(n-d-1) & 0 \\
	\end{array} } \right).\]
	By a direct calculation, the characteristic polynomial of $Q_{\Pi_1}$ is 
	$$g(x)=x^3-(2n-6)x^2-(nd^2-5d^2+2nd-2d+5n-9)x-2(d-1)^2(n-1).$$
	By Lemma \ref{lem:quo-spec}, all eigenvalues of $Q_{\Pi_1}$ are 
	eigenvalues of $N$ as well. Since $g(-2)\neq 0$, the eigenvalues of $N$ are $-2$ 
	with multiplicity $n-3$, and the roots of $g(x)=0$.  This completes the proof.   
	\end{proof}
	
	We proceed to give our proof of Theorem \ref{thm:unimodal-log-conc-min4pc}.

	\textbf{Proof of Theorem \ref{thm:unimodal-log-conc-min4pc}}: 
	For a tree $T$ of order $n$, by Theorem \ref{thm:inertia-min4pc}, 
	we have $\Inertia(\Min_T[B,B])=(1,n-1,0)$. Hence, by 
	Corollary \ref{cor:dist-matrix-unim}, the sequence $|a_0|, |a_1|, 
	\cdots,|a_{n-2}|$ is unimodal and log-concave.
	
Now, we have to find the peak location of this unimodal sequence. By 
Theorem \ref{spectra-min4pc}, it follows that the characteristic polynomial 
of $\Min_T[B,B]$ is 
	$i\displaystyle f(x)=(x+2)^{n-3}(x^3+b_1x^2+c_1x+d_1),$
	where $b_1=-(2n-6)$, $c_1=-(nd^2-5d^2+2nd-2d+5n-9)$ and $d_1=-2(d-1)^2(n-1)$. 
	%Multiplying the binomial expansion of $(x+2)^{n-3}$ by $(x^3+b_1x^2+c_1x+d_1)$, 
	%and combining the coefficients of the same power, we have 
	%
	%\begin{align*}
	% f(x)&=\bigg[d_1\binom{n-3}{0}2^{n-3}\bigg]x^0+\bigg[2c_1\binom{n-3}{0}+d_1\binom{n-3}{1}\bigg]2^{n-4}x^1\\
	% &+\bigg[4b_1\binom{n-3}{0}+2c_1\binom{n-3}{1}+d_1\binom{n-3}{2}\bigg]2^{n-5}x^2\\
	% &+ \sum_{k=3}^{n-3} \bigg[8\binom{n-3}{k-3}+4b_1\binom{n-3}{k-2}+2c_1\binom{n-3}{k-1}+d_1\binom{n-3}{k}\bigg]2^{n-k-3}x^k\\
	% &+\bigg[4\binom{n-3}{n-5}+2b_1\binom{n-3}{n-4}+c_1\binom{n-3}{n-3}\bigg]x^{n-2}\\
	%  &+\bigg[2\binom{n-3}{n-4}+b_1\binom{n-3}{n-3}\bigg]x^{n-1}+\binom{n-3}{n-3}x^n
	%\end{align*}
	Let $a_k$ be the coefficient of $x^k$ in $f(x)$.  One can check that 
	%Therefore, from the above expression, we have
	\begin{align*}
	a_0&=d_1\binom{n-3}{0}2^{n-3}=-2(d-1)^2(n-1)2^{n-3}, \\ a_1&=\bigg[2c_1\binom{n-3}{0}+d_1\binom{n-3}{1}\bigg]2^{n-4}\\
	&=-\big[2(nd^2-5d^2+2nd-2d+5n-9)+2(d-1)^2(n-1)(n-3)\big]2^{n-4},\\ 
	a_2&=\Big[4b_1\binom{n-3}{0}+2c_1\binom{n-3}{1}+d_1\binom{n-3}{2}\Big]2^{n-5}\\
	&=-\big[8(n-3)+2(nd^2-5d^2+2nd-2d+5n-9)(n-3)\big]2^{n-5} \\
	&~~~+2^{n-5}\big[(d-1)^2(n-1)(n-3)(n-4)\big],\\
	a_{n-2}&=4\binom{n-3}{n-5}+2b_1\binom{n-3}{n-4}+c_1\binom{n-3}{n-3}\\
	&=-\big[2(n-3)(n-4)+(2n-6)(n-3)+(nd^2-5d^2+2nd-2d+5n-9)\big],\\
	a_{n-1}&=2\binom{n-3}{n-4}+b_1\binom{n-3}{n-3}=2(n-3)-2(n-3)=0,~~ a_n=1,
	\end{align*}
	and for $3\leq k\leq n-3$
	\begin{align*}
	a_k &=\bigg[8\binom{n-3}{k-3}+4b_1\binom{n-3}{k-2}+2c_1\binom{n-3}{k-1}+d_1\binom{n-3}{k}\bigg]2^{n-k-3} \\
	%&= \binom{n-3}{k-3}2^{n-k-3}\bigg(8+\frac{4b_1(n-k)}{k-2}+ \frac{2c_1(n-k)(n-k-1)}{(k-2)(k-1)}+\frac{d_1(n-k)(n-k-1)(n-k-2)}{(k-2)(k-1)k}\bigg)\\
	&= \binom{n-3}{k-3}2^{n-k-3} f_1(n,k), ~~ \text{where}\\
	f_1(n,k)&=8+\frac{4b_1(n-k)}{k-2}+ \frac{2c_1(n-k)(n-k-1)}{(k-2)(k-1)}+\frac{d_1(n-k)(n-k-1)(n-k-2)}{(k-2)(k-1)k}\\
	&=8-\frac{8(n-3)(n-k)}{k-2}- \frac{2(nd^2-5d^2+2nd-2d+5n-9)(n-k-1)(n-k)}{(k-1)(k-2)}\\
	&~~~~-\frac{2(d-1)^2(n-1)(n-k-2)(n-k-1)(n-k)}{k(k-1)(k-2)}.
	\end{align*}
	Thus, $|a_k|=\binom{n-3}{k-3}2^{n-k-3} |f_1(n,k)|$. When $n \geq 8$,
	it is easy to check that $|a_0|\leq |a_1|\leq |a_2|$ and $a_{n-3}\geq a_{n-2}$.  
	Further, when $3\leq k\leq n-3$, we have 
	\begin{align*}
	|a_k|- |a_{k-1}|
	&= \binom{n-3}{k-3}2^{n-k-3} |f_1(n,k)|-\binom{n-3}{k-4}2^{n-k-2} |f_1(n,k-1)|\\
	%&= \binom{n-3}{k-4}2^{n-k-3} \bigg[\bigg(\frac{n-k+1}{k-3}\bigg)|f_1(n,k)|-2|f_1(n,k-1)|\bigg]\\
	&= \binom{n-3}{k-4}2^{n-k-3} \bigg[\frac{8(n-2)(n^2-4kn+4n+3k^2-4k-1)}{(k-3)(k-2)}\\
	&~~~~~~~~+\frac{2(nd^2-5d^2+2nd-2d+5n-9)(n-k)(n-k+1)}{(k-2)(k-3)}\cdot \bigg(\frac{n-3k+1}{k-1}\bigg)\\
	&~~~~~~~~+\frac{2(d-1)^2(n-1)(n-k-1)(n-k)(n-k+1)}{(k-1)(k-2)(k-3)}\cdot \bigg(\frac{n-3k-2}{k}\bigg)\bigg].
	\end{align*}
	Hence, when $3\leq k\leq n-3$, one can verify that
	$|a_k|\geq |a_{k-1}|$ if and only if $k\leq \frac{n-2}{3}$ and 
	$|a_k|\leq |a_{k-1}|$ if and only if $k\geq \frac{n+4}{3}$. 
	Thus, when $n \geq 8$, we have  $|a_0|\leq |a_1|\leq |a_2| \leq \hdots 
	\leq |a_{\lfloor \frac{n-2}{3} \rfloor}|$ and 
	$|a_{\lceil \frac{n+4}{3}-1\rceil}| \geq |a_{\lceil \frac{n+4}{3}\rceil}| 
	\geq \hdots \geq |a_{n-3}|\geq |a_{n-2}|$. Hence, if $|a_t|=\max_{0\leq k\leq n-2} 
	|a_k|$, then $ \lfloor \frac{n-2}{3} \rfloor \leq t \leq \lceil \frac{n+1}{3}\rceil$. 
	This completes the proof.

	\section{Isometrically embedding $\Min_T$ in $\ell_1$ space}
	\label{sec:ell1}
	
	For any tree $T$ having $n$ vertices, we show that the $\Min_T$ matrix is
	isometrically embeddable in $\RR^{n-1}$ equipped with the $\ell_1$ norm.
	This gives an alternate proof that the  matrix $\Min_T$ has 
$r-1$ negative eigenvalues and one positive eigenvalue, where $r$ is the
	rank of $\Min_T$.
	
	Identify the $(n-1)$ dimensions of $\RR^{n-1}$ with the edges of $T$.  For $\{i,j\} \in 
	\VV$, the embedding  $\phi_{ \{i,j\} }$  maps $\{i,j\}$ to the 
	incidence vector of the unique path $P_{i,j}$ between $i$ and $j$ in $T$.
	We illustrate by an example.  Let $T$ be the tree given in Figure 
	\ref{fig:tree1} with edge set $E = \{e_1,e_2,e_3, e_4\}$.  
	For brevity, for $\{i,j\} \in \VV$, we omit the comma in the subscript 
	and denote $\phi_{i,j}$ in Figure \ref{fig:tree1} as $\phi_{ij}$.
	Let $f = \{1,4\} \in \VV$.  The set of edges on the path $P_{1,4}$ 
	between the vertices $1$ and $4$ is clearly $P_{1,4} = \{e_1,e_2 \}$ 
	and thus, the column vector $\phi_{14} = (1,1,0,0)^t$.   This column 
	vector $\phi_{1,4}$ is illustrated with a different colour 
	in Figure \ref{fig:tree1}.
	%For concreteness, let $B = E \cup \{f\}$. 
	
	\begin{figure}[h]
		\begin{minipage}{0.38\textwidth}
			\centerline{\includegraphics[scale=0.65]{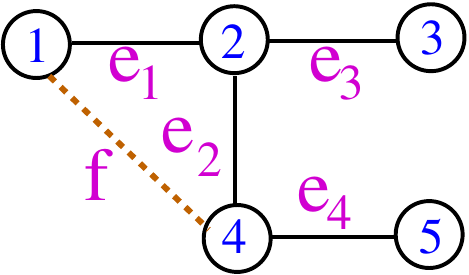}}
		\end{minipage}
		\begin{minipage}{0.6\textwidth}
			$\begin{array}{l|r|r|r|r|r| r|r|r|r|r|}
				& \phi_{12} & \phi_{13} & \phi_{14} & \phi_{15} & \phi_{23} & \phi_{24} & \phi_{25} & 
				\phi_{34} & \phi_{35} & \phi_{45} \\ \hline
				e_1 & 1 & 1 & \fby{1} & 1 &    0 & 0 & 0 &  0 & 0 &  0 \\
				e_2 & 0 & 0 & \fby{1} & 1 &    0 & 1 & 1 &  1 & 1 &  0 \\
				e_3 & 0 & 1 & \fby{0} & 0 &    1 & 0 & 0 &  1 & 1 &  0 \\
				e_4 & 0 & 0 & \fby{0} & 1 &    0 & 0 & 1 &  0 & 1 &  1 \\
			\end{array}$
		\end{minipage}
		\caption{A tree and its embedding.  Column $\phi_{14}$ is illustrated on the left.
		}
		\label{fig:tree1}
	\end{figure}
	
	%in Figure \ref{fig:tree1}.  
	%We now check for all $x,y \in B$, 
	%that $\Min_T(x,y) = \lVert \phi(x)-\phi(y) \rVert_1$.
	%It is clear that $\lvert \phi(e_r)- \phi(e_s) \rVert_1 = 2$
	%when $r \not= s$.  Further, recall $f = \{i,j\}$ with $\{i,j\} 
	%\not\in E$ and let $P_{i,j}$ be the path between $i$ and $j$ in $T$.  
	%It is easy to see that 
	%$\lVert \phi(f) - \phi(e_s)\rVert_1 =
	%\begin{cases}
	%1 \mbox{ if $e_s \in P_{i,j}$}\\
	%3 \mbox{ if $e_s \not\in P_{i,j}$}\\
	%\end{cases}
	%$ 
	
	%Note that in $n-1$ dimensions by tacitly identifying dimensions 
	%with the edges, for $x \in B$, 
	%$\phi(x)$ is the indicator vector of the edges
	%on the path between the end-points of $x$.  
	
	%In \cite[Theorem 5]{ali-jana-nagar-siva_4pc-edge_weighted}, 
	%it is proved that any set $B$ 
	%of the form $B = E \cup \{f\}$ where $f \not \in E$ is a basis
	%for the row space of $\Min_T$.
	
	%\begin{theorem}
	%\label{thm:l1_embedding}
	%Let $T$ be a tree on $n$ vertices. 
	%and let $B = E \cup \{f\}$
	%where $f = \{i,j\}$ and $f \not\in E$.  Let $\Min_T[B,B]$ be 
	%the matrix $\Min_T$ restricted to the rows and columns in $B$. 
	%Its matrix $\Min_T$ is isometrically $\ell_1$-embeddable in 
	%$\RR^{n-1}$.
	%\end{theorem}
	
	After seeing the example in Fig \ref{fig:tree1}, we are 
	now ready for the proof of Theorem \ref{thm:l1_embedding}.
	
	\begin{proof} (Of Theorem \ref{thm:l1_embedding}):
		We identify the $n-1$ dimensions of $\RR^{n-1}$ with 
		the edges of $T$.  
		Consider the embedding $\phi: \VV \rightarrow \RR^{n-1}$ 
		described above.  Thus $\phi_{i,j}$ is the incidence vector 
		of the path $P_{i,j}$.
		%For $x \in B$, let $x = \{i,j\}$ and let 
		%$P_{i,j}$ be the unique path between $i$ and $j$ in $T$.
		%Let $\phi(x)$ be the indicator vector of the edges on $P_{i,j}$.
		For all $i,j, s,t \in V(T)$, we show that 
		$\Min_T(\{i,j\}, \{s,t\}) = 
		\lVert \phi_{i,j} - \phi_{s,t} \rVert_1$.
		
		%It is clear that $\lVert \phi(e_r) - \phi(e_s)\rVert_1 = 
		We consider two cases depending on whether the path
		$P_{i,j}$ intersects the path $P_{s,t}$.
		
		{\bf Case 1, (when $P_{i,j} \cap P_{s,t} = \emptyset$): }  In this case,
		note that 
		$\lVert \phi_{i,j} - \phi_{s,t} \rVert_1 = d_{i,j} + d_{s,t}$.  
		If $\alpha$ is a vertex lying on $P_{i,j}$ and $\beta$ is a vertex on
the path $P_{s,t}$  are chosen such that $d_{\alpha,\beta}$ is the smallest
among choices of vertices $\alpha$ on the path $P_{i,j}$ and $\beta$ on
the path $P_{s,t}$, then as $d_{\alpha,\beta} \geq 0$, we have 
$d_{i,j} + d_{s,t} \leq d_{i,t} + d_{j,s}$ and $d_{i,j} + d_{s,t} 
\leq d_{i,s} + d_{j,t}$.  Thus, 
$\Min_T(\{i,j\}, \{s,t\}) = d_{i,j} + d_{s,t}= 
		\lVert \phi_{i,j} - \phi_{s,t} \rVert_1$.
		
		{\bf Case 2, (when $P_{i,j} \cap P_{s,t} \not= \emptyset$): }  Let 
		$S = P_{i,j} \cap P_{s,t}$.  As $T$ is a tree, it is easy to see 
		that $S$ is a set of edges
		on a path from $\alpha$ to $\beta$, where  $\alpha, \beta \in V(T)$.
		That is $d_{\alpha,\beta} = |S|$.  In this case, it is easy to see that 
		$\lVert \phi_{i,j} - \phi_{s,t} \rVert_1 = d_{i,j} + d_{s,t} - 
		d_{\alpha,\beta}$.   It is now clear that the minimum element 
		of the set $S_{i,j,s,t}$ is
		$d_{i,j} + d_{s,t} - d_{\alpha,\beta}$ completing the proof.
		%$P_{i,j}
		%$\begin{cases}
			%0 & \mbox{if $r = s$},\\
			%2 & \mbox{if $r \not= s$}
			%\end{cases}$.
			%Further, if $f \in B$ with $f \not \in E$, and $f = \{i,j\}$, then
			%it is clear that $\lVert \phi(f) - \phi(e_s)\rVert_1 = 
			%\begin{cases}
			%d_{i,j} - 1 & \mbox{if $e_s \in P_{i,j}$},\\
			%d_{i,j} + 1 & \mbox{if $e_s \not\in P_{i,j}$},
			%\end{cases}$.
		\end{proof}
		
		We make two remarks from the proof of Theorem 
		\ref{thm:l1_embedding}.
		
		\begin{remark}
			\label{rem:hypercube_embedding}
			In the proof of Theorem  \ref{thm:l1_embedding}, note that the 
			images $\phi_{i,j}$ are vectors in $\{0,1\}^{n-1}$.  Thus, for
			any tree $T$ having $n$ vertices, its $\Min_T$ matrix is isometrically
			embeddable in the $(n-1)$ dimensional hypercube equipped with
			the Hamming metric.  This is easily seen to be stronger than 
			being isometrically embeddable in $\ell_1$ space.
		\end{remark}
		
		\begin{remark}
			\label{rem:triangle-ineq}
			Theorem \ref{thm:l1_embedding}
			shows that the $\Min_T$  matrix satisfies triangle inequality.  
		\end{remark}
		
		%A distance matrix $D = (d_{i,j})_{1 \leq i,j \leq n}$ is said 
		%to be a {\sl hypermetric} if $\sum_{1 \leq i < j \leq n} 
		%x_ix_jd_{i,j} \leq 0$ for all $x \in \ZZ^n$ with 
		%$\sum_{i=1}^n x_i = 1$.  If the above inequality holds for 
		%all $x \in \ZZ^n$ with $\sum_{i=1}^n x_i = 0$, then $D$ is 
		%said to be a {\sl negative type metric.}
		
		The following corollary is easily follows from Theorem \ref{thm:l1_embedding}.

		\begin{corollary}
			\label{cor:inertia_min4pc}
			For any tree $T$, the matrix $\Min_T$ is {\sl hypermetric,} is of {\sl negative
				type} and has exactly one positive eigenvalue.
		\end{corollary}

	\section{The 2-Steiner distance matrix $\DD_2(T)$ of a tree $T$}
	\label{sec:steiner}
Recall that for a tree $T$ having $n$ vertices and $p$ pendant vertices,  
Azimi and Sivasubramanian in \cite[Theorem 1]{aliazimi-siva-steiner-2-dist} 
showed that its 2-Steiner distance matrix $\DD_2(T)$ has rank 
$r = 2n-p-1$.  They also gave the following basis.

\begin{remark}\cite[Remark 6]{aliazimi-siva-steiner-2-dist}
\label{remark: basis-2-steiner}
For a tree $T$ of order $n$ with $p$ leaves, let $B_1,B_2,\hdots,B_{n-p}$ be the blocks 
of its line graph $\LG(T)$ such that $|B_i|=b_i$ for $ i=1,\ldots,n-p$. If $e_i\in V(\LG(T))$, 
$i=1,\ldots,n-1$ and $f_j$, $j=1,\ldots,n-p$, is the symmetric difference 
of endpoints of edge  $f_j^\prime \in B_j$ in $\LG(T)$, then 
$B=\{e_1,e_2,\hdots,e_{n-1},f_1,\hdots,f_{n-p}\}$ forms a 
basis for the row space of $\DD_2(T)$.
\end{remark}
	
Below, we provide the proof of the first part of Theorem \ref{thm:unimodal-log-conc-2steiner-dist}, followed by Corollary \ref{cor:dist-matrix-unim}, which shows that the sequence $|a_0|, |a_1|, \hdots, |a_{r-1}|$ is unimodal and log-concave.
	
\begin{proof} (Of Theorem \ref{thm:unimodal-log-conc-2steiner-dist} :) 
Let $T$ be a tree of order $n$ with $p$ pendant vertices and $r = 2n-p-1$. 
Azimi and Sivasubramanian in \cite[Theorem 18]{aliazimi-siva-steiner-2-dist}
showed that the matrix $\DD_2(T)[B,B]$ has $r-1$ negative eigenvalues and 
one positive eigenvalue. Hence, by Corollary \ref{cor:dist-matrix-unim}, 
it follows that the sequence $|a_0|, |a_1|, \hdots,|a_{r-1}|$ is unimodal and log-concave,
		completing the first part.
	\end{proof}
	
	For the second part of Theorem \ref{thm:unimodal-log-conc-2steiner-dist}, we
give our bounds on the peak location of the coefficients of 
$\charpoly_{\DD_2(T)[B,B]}(x)$.  As we consider  
	three families of trees, the star $S_n$, the bi-star $S_{1,n-3}$
	and the path $P_n$ on $n$ vertices, we trifurcate our proof into
	three subsections.

	 %%%%%%%%%%%%%%%%%%%%%%%%%%%%%%%%%%%%%%%%%%%%%%%%%%%%%%%%%%%%%%%%%%%%%%%%%%%%%%
\subsection{Peak location for star trees}
\label{subsec:peak-star}
For a star $S_n$ on $n$ vertices, $B=E\cup \{f\}$ is a basis of $\DD_2(S_n)$, where $E$ is the edge set of $S_n$ and $f=\{i,j\}\notin E$ for any two vertices $i,j$ of $S_n$. In the following theorem, we find the spectra of $\DD_2(S_n)[B,B]$.
	  
\begin{theorem}
\label{spectra-steiner-star}
For a star $S_n$ on $n$ vertices,
the eigenvalues of 
$\DD_2(S_n)[B,B]$ are $-1$ with multiplicity $n-3$
and the roots of the cubic polynomial $g(x)=x^3-2(n-1)x^2-7(n-2)x-(n-1).$
\end{theorem}
	 \begin{proof}
Let $S_n$ have vertex set $V=\{1,2,\hdots,n\}$.  Let $E(T)=\{e_i = 
\{1,i+1\}: 1 \leq i < n\}$ be its edge set. Without loss of generality, assume that 
$f=\{2,3\}\notin E(T)$ and $B=\{e_1,e_2,\hdots,e_{n-1},f\}$.  
		 Clearly, 
		 \[ 
\DD_2(S_n)[B,B] =
\begin{blockarray}{cccccc}
			 & e_1 & e_2 & \hdots & e_{n-1} & f\\
			 \begin{block}{c(ccccc)}
				 e_1 & 1 & 2  &  \hdots & 2 & 2 \\
				 e_2 & 2 & 1  &  \hdots & 2 & 2 \\
%		 			e_3 & 2 & 2  & \hdots & 2 & 3 \\
				 \vdots & \vdots & \vdots  &\vdots &\ddots &\vdots & \vdots \\
				 e_{n-1} & 2 & 2  & \hdots & 1 & 3 \\
				 f & 2 & 2  & \hdots & 3 & 2\\
			 \end{block}
		 \end{blockarray} .\]
		 
Let $e(i,j)$ be the $n$-dimensional column vector with its $i$-th component 
$1$, its $j$-th component $-1$ and all other components being $0$. If 
$X=\{e(1,2)\}\cup \{e(j,j+1):3\leq j\leq n-2\}$, then for 
any $\textbf{x}\in X$, 
we clearly have $\big(\DD_2(S_n)[B,B]\big)\textbf{x}=-\textbf{x}$. 
Note that $|X|=n-3$, 
and all vectors in $S$ are linearly independent. Therefore, $-1$ is an 
eigenvalue of $\DD_2(S_n)[B,B]$ with multiplicity at least $n-3$.
		 
Let $E_1=\{e_1,e_2\}$, $E_2=\{e_3,\hdots,e_{n-1}\}$ and
$\Pi_2: E_1 \cup E_2 \cup \{f\}$.  It is easy to see that 
$\Pi_2$ an equitable 
partition of $\DD_2(S_n)[B,B]$ and gives rise to the quotient matrix 
\[Q_{\Pi_2}=
\left ( {\begin{array}{ccc}
3 & 2(n-3) & 2\\
4 & 2(n-4)+1 & 3 \\
4 & 3(n-3) & 2 \\
\end{array} } \right).
\mbox{ A simple computation gives }
\]
		 \begin{equation}
			 \label{eqn:charpoly_equi_partn}
			 \charpoly_{Q_{\Pi_2}}(x)  =  g(x)=x^3-2(n-1)x^2-7(n-2)x-(n-1).
		 \end{equation}
By Lemma \ref{lem:quo-spec}, all eigenvalues of $Q_{\Pi_2}$ are eigenvalues 
of $\DD_2(S_n)[B, B]$. Since $g(-1)\neq 0$, the eigenvalues of 
$\DD_2(S_n)[B, B]$ are $-1$ with multiplicity $n-3$, and the roots of $g(x)=0$,
completing the proof.
\end{proof}

\noindent
In our next result, we determine the peak location of the coefficients 
of $\charpoly_{\DD_2(S_n)[B,B]}(x)$ up to an interval of constant 
size that is independent of $n$.  
	 
\begin{theorem}
\label{peak-steiner-star}
Let $B$ be the basis of $\DD_2(S_n)$ used
in Theorem \ref{spectra-steiner-star}. If $a_0,a_1,\hdots,a_n$ are the 
coefficients of the characteristic polynomial of $\DD_2(S_n)[B,B]$ and 
$|a_t|=\max |a_k|$, then $\lfloor \frac{n-2}{2} \rfloor \leq t 
\leq \lceil \frac{n}{2}\rceil$.
\end{theorem}
\begin{proof}
By Theorem \ref{spectra-steiner-star} and 
\eqref{eqn:charpoly_equi_partn},  we have 
%if $f(x) = \charpoly_{\DD_2(S_n)[B,B]}(x)$, then 
%	\begin{align*}
$\charpoly_{\DD_2(S_n)[B,B]}(x)=(x+1)^{n-3}\Big(x^3-2(n-1)x^2-7(n-2)x-(n-1)\Big).$
			 %&=-\bigg[(n-1)+(n^2+3n-11)x+\frac{1}{2}(n^3+6n^2-47n+68)x^2 \\
			 %&~~~~~~~~~+\sum_{k=3}^{n-3} \bigg[-\binom{n-3}{k-3}+2(n-1)\binom{n-3}{k-2}+14(n-2)\binom{n-3}{k-1}+(n-1)\binom{n-3}{k}\bigg]x^k\\
			 %&~~~~~~~~~+(3n^2+5n-28)x^{n-2}+(n+1)x^{n-1}\bigg]+x^n.
%		 	\end{align*}
If $a_k$ is the coefficient of $x^k$ in $f(x)$, then it is easy to see that
\begin{align*}
a_0 &=-(n-1),~~ a_1 =-(n^2+3n-11),~~a_2=-\frac{1}{2}(n^3+6n^2-47n+68),\\
a_k &= -\bigg[-\binom{n-3}{k-3}+2(n-1)\binom{n-3}{k-2}+14(n-2)\binom{n-3}{k-1}+(n-1)\binom{n-3}{k}\bigg]\\
&~~~~~\text{when}~ 3\leq k\leq n-3,\\
a_{n-2}&=-(3n^2+5n-28),~~ a_{n-1}=-(n+1), ~~\text{and} ~~ a_n=1.
\end{align*}
It is easy to check that $|a_0|\leq |a_1|\leq |a_2|$ and $|a_{n-2}|\geq 
|a_{n-1}|\geq |a_{n}|$ when $n\geq 6$.  When $4\leq k\leq n-3$, one can check that
\begin{align*}
			 &|a_k|- |a_{k-1}|\\
			 &= \binom{n-3}{k-4} \bigg[\frac{(2n^3+4n^2-6kn^2+4k^2n-3kn-2k^2+4)}{(k-2)(k-3)}+\bigg(\frac{7(n-2)(n-k)(n-k+1)}{(k-1)(k-2)}\bigg)\cdot\\
			 &~~~~~~~~~~~~~~~~~~~~~~\bigg(\frac{n-2k+2}{k-3}\bigg)+\bigg(\frac{(n-1)(n-k-1)(n-k)(n-k+1)}{(k-1)(k-2)(k-3)}\bigg)\cdot \bigg(\frac{n-2k-2}{k}\bigg)\bigg].
\end{align*}
Hence, when $4\leq k\leq n-3$, it is easy to verify that
$|a_k|\geq |a_{k-1}|$ if and only if $k\leq \frac{n-2}{2}$ and 
$|a_k|\leq |a_{k-1}|$ if and only if $k\geq \frac{n+2}{2}$. Thus, 
we have  $|a_0|\leq |a_1|\leq |a_2| \leq \hdots \leq 
|a_{\lfloor \frac{n-2}{2} \rfloor}|$ and 
$|a_{\lceil \frac{n+2}{2}-1\rceil}| \geq |a_{\lceil \frac{n+2}{2}\rceil}| 
\geq \hdots \geq |a_{n-3}|\geq |a_{n-2}|$.  
Hence, if $|a_t|=\max_{0\leq k\leq n-2} |a_k|$, then $\lfloor \frac{n-2}{2} 
\rfloor \leq t \leq \lceil \frac{n}{2}\rceil$, completing our proof.
\end{proof}

\subsection{Peak location for the bi-star $S_{1,n-3}$}
\label{subsec-2-steiner-bistar}
Let $S_{1,n-3}$ be a tree on $n$ vertices obtained from $P_2$ that has
the edge $\{v_1,v_2\}$ by 
attaching a pendant vertex $v_0$ to $v_1$ and $(n-3)$ pendant vertices 
$v_3,v_4,\hdots,v_{n-1}$ to $v_2$. Let $e_1=\{v_0,v_1\}, 
e_2=\{v_1,v_2\}$ and $e_i=\{v_2,v_i\}$ for $3\leq i \leq n-1$. 
Since $S_{1,n-3}$ has $n-2$ pendant vertices, two 
types of basis $B_1$ and $B_2$  are output by  the algorithm given by
Azimi and Sivasubramanian (see Remark \ref{remark: basis-2-steiner}).  
These are 
\begin{align*}
B_1&=\{e_1,e_2,\hdots,e_{n-1},f_1,f_2\}~ \text{where}~ f_1=\{v_0,v_2\},f_2=\{v_1,v_3\}~\text{and}\\
		 B_2&=\{e_1,e_2,\hdots,e_{n-1},f_1,f_2\}~ \text{where}~ f_1=\{v_0,v_2\},f_2=\{v_3,v_4\}.
	 \end{align*}
	 
We find the eigenvalues of both $\DD_2(S_{1,n-3})[B_1,B_1]$ and 
$\DD_2(S_{1,n-3})[B_2,B_2]$.
	 
\begin{theorem}
\label{spectra-steiner-bistar}
Let $B_1$ and $B_2$ be the 
bases of $S_{1,n-3}$ as mentioned above. Then,
\begin{enumerate}
\item the eigenvalues of $\DD_2(S_{1,n-3})[B_1,B_1]$ are $-1$ with multiplicity 
$n-5$ and the roots of the polynomial 
			 $h_1(x)=x^6-2(n-1)x^5-3(7n-12)x^4-18(3n-7)x^3-5(9n-22)x^2-(13n-28)x-(n-1)$.
			 \item the eigenvalues of $\DD_2(S_{1,n-3})[B_2,B_2]$ are $-1$ with multiplicity 
			 $n-5$ and the roots of the polynomial 
			 $h_2(x)=x^6-2(n-1)x^5-(21n-22)x^4-(62n-141)x^3-(53n-133)x^2-(15n-34)x-(n-1)$.
		 \end{enumerate}
	 \end{theorem}
	 
\begin{proof}
With the given labelling, we have 
		 
\begin{align*}
\DD_2(S_{1,n-3})[B_1,B_1]&=\begin{blockarray}{cccccccc}
%& e_1 & e_2 & e_3 & e_4 & \hdots & e_{n-1} & f_1 & f_2\\
& e_1 & e_2 & e_3 &  \hdots & e_{n-1} & f_1 & f_2\\
\begin{block}{c(ccccccc)}
e_1 & 1 & 2 & 3 &  \hdots & 3 & 2 & 3 \\
e_2 & 2 & 1 & 2 &  \hdots & 2 & 2 & 2 \\
e_3 & 3 & 2 & 1 &  \hdots & 2 & 3 & 2 \\
%e_4 & 3 & 2 & 2 &  \hdots & 2 & 3 & 3 \\
\vdots & \vdots & \vdots &\vdots &\ddots &\vdots &\vdots & \vdots \\
e_{n-1} & 3 & 3 & 2  &  \hdots & 1 & 3 & 3 \\
f_1 & 2 & 2 & 3 &  \hdots & 3 & 2 & 3 \\
f_2 & 3 & 2 & 2 &  \hdots & 3 & 3 & 2 \\
\end{block}
\end{blockarray}\quad \text{and}\\
\DD_2(S_{1,n-3})[B_2,B_2]&=\begin{blockarray}{ccccccccc}
%& e_1 & e_2 & e_3 & e_4 & e_5 & \hdots & e_{n-1} & f_1 & f_2\\
& e_1 & e_2 & e_3 & e_4 &  \hdots & e_{n-1} & f_1 & f_2\\
\begin{block}{c(cccccccc)}
e_1 & 1 & 2 & 3 & 3 &  \hdots & 3 & 2 & 4 \\
e_2 & 2 & 1 & 2 & 2 &  \hdots & 2 & 2 & 3 \\
e_3 & 3 & 2 & 1 & 2 &  \hdots & 2 & 3 & 2 \\
e_4 & 3 & 2 & 2 & 1 &  \hdots & 2 & 3 & 2 \\
%e_5 & 3 & 2 & 2 & 2 &  \hdots & 2 & 3 & 3 \\
\vdots & \vdots & \vdots &\vdots &\vdots &\ddots &\vdots &\vdots & \vdots \\
e_{n-1} & 3 & 2 & 2  & 2 &  \hdots & 1 & 3 & 3 \\
f_1 & 2 & 2 & 3 & 3 &  \hdots & 3 & 2 & 4 \\
f_2 & 4 & 3 & 2 & 2 &  \hdots & 3 & 4 & 2 \\
\end{block}
\end{blockarray}~.
		 \end{align*}
		 
As before, let $e(i,j)$ be the $n$-dimensional column vector with its 
$i$-th component $1$, its $j$-th component $-1$ and all other 
components being $0$.  If $X= \{e(j,j+1):4\leq j\leq n-2\}$ and  
$Y=\{e(3,4)\}\cup \{e(j,j+1):5\leq j\leq n-2\}$, then for any $\textbf{x}\in X$ 
and $\textbf{y}\in Y$, we have $\big(\DD_2(S_{1,n-3})[B_1,B_1]\big)\textbf{x}
=-\textbf{x}$ and $\big(\DD_2(S_{1,n-3})[B_2,B_2]\big)\textbf{y}=-\textbf{y}$. 
Note that $|X|=|Y|=n-3$, and that all vectors in both $X$ and $Y$ 
are linearly independent.  
%Similarly, it is easy to see that all vectors in $Y$ are also 
%linearly independent.
Therefore, $-1$ is an eigenvalue of both $\DD_2(S_{1,n-3})[B_1,B_1]$ and 
$\DD_2(S_{1,n-3})[B_2,B_2]$ with multiplicity at least $n-3$.

If $E_1=\{e_4,\hdots,e_{n-1}\}$, then it is easy to see that 
$\Pi_3:\{e_1\}\cup \{e_2\}\cup \{e_3\}\cup E_1 \cup \{f_1\}\cup \{f_2\}$ is an 
equitable partition of $\DD_2(S_{1,n-3})[B_1,B_1]$ with the quotient matrix given below.
%\[
%\]

A simple computation gives the characteristic polynomial of $Q_{\Pi_3}$ to be
$h_1(x)=x^6-2(n-1)x^5-3(7n-12)x^4-18(3n-7)x^3-5(9n-22)x^2-(13n-28)x-(n-1).$
By Lemma \ref{lem:quo-spec}, the eigenvalues of $Q_{\Pi_3}$ are eigenvalues 
of $\DD_2(S_{1,n-3})[B_1,B_1]$. Since $h_1(-1)\neq 0$, the eigenvalues of 
$\DD_2(S_{1,n-3})[B_1,B_1]$ are $-1$ with multiplicity $n-5$, and the roots of $h_1(x)=0$.
		 
If $E_2=\{e_3,e_4\}$ and $E_3=\{e_5,\hdots,e_{n-1}\}$, then it is easy to see that 
$\Pi_4:\{e_1\}\cup \{e_2\}\cup E_2 \cup E_3 \cup \{f_1\}\cup \{f_2\}$ is an 
equitable partition of $\DD_2(S_{1,n-3})[B_2,B_2]$ with the quotient matrix 
given below.
\[
Q_{\Pi_3}=
\left ( {\begin{array}{cccccc}
1 & 2 & 3 & 3(n-4) & 2 & 3\\
2 & 1 & 2 & 2(n-4) & 2 & 2\\
3 & 2 & 1 & 2(n-4) & 3 & 2\\
3 & 2 & 2 & 2(n-4)-1 & 3 & 3\\
2 & 2 & 3 & 3(n-4) & 2 & 3\\
3 & 2 & 2 & 3(n-4) & 3 & 2\\
\end{array} } \right)
\mbox{ and }
Q_{\Pi_4}=
		 \left ( {\begin{array}{cccccc}
				 1 & 2 & 6 & 3(n-5) & 2 & 4\\
				 2 & 1 & 4 & 2(n-5) & 2 & 3\\
				 3 & 2 & 3 & 2(n-5) & 3 & 2\\
				 3 & 2 & 4 & 2(n-5)-1 & 3 & 3\\
				 2 & 2 & 6 & 3(n-5) & 2 & 4\\
				 4 & 3 & 4 & 3(n-5) & 4 & 2\\
		 \end{array} } \right)\]
The characteristic polynomial of $Q_{\Pi_4}$ clearly equals 
$h_2(x)=x^6-2(n-1)x^5-(21n-22)x^4-(62n-141)x^3-(53n-133)x^2-
(15n-34)x-(n-1).$  By Lemma \ref{lem:quo-spec}, all eigenvalues of 
$Q_{\Pi_4}$ are eigenvalues of $\DD_2(S_{1,n-3})[B_2,B_2]$. Since 
$h_2(-1)\neq 0$, the eigenvalues of $\DD_2(S_{1,n-3})[B_2,B_2]$ are 
$-1$ with multiplicity $n-5$, and the roots of $h_2(x)=0$.
\end{proof}
	 
%		 We next determine the peak location of coefficients of the
%		 characteristic polynomial for $\DD_2(S_{1,n-3})[B_1,B_1]$ and 
%$\DD_2(S_{1,n-3})[B_2,B_2]$.
	 
For both $\DD_2(S_{1,n-3})[B_1,B_1]$ and 
$\DD_2(S_{1,n-3})[B_2,B_2]$, we determine the peak location of 
coefficients of their characteristic polynomial in the next result. 
	 
\begin{theorem}
\label{peak-location-steiner-bistar}
Let $S_{1,n-3}$ be the tree on $n$ vertices as mentioned above.
\begin{enumerate}
\item Let $a_0,a_1,\hdots,a_n, a_{n+1}$ be the coefficients of 
$\charpoly_{\DD_2(S_{1,n-3})[B_1,B_1]}(x)$  and $|a_t|=\max |a_k|$. 
Then $\lfloor \frac{n-4}{2} \rfloor \leq t \leq \lceil \frac{n+4}{2}\rceil$.

\item Let $b_0,b_1,\hdots,b_n,b_{n+1}$ be the coefficients of %the 
$\charpoly_{\DD_2(S_{1,n-3})[B_2,B_2]}(x)$  and 
$|b_t|=\max |b_k|$. Then $\lfloor \frac{n-4}{2} \rfloor \leq t 
\leq \lceil \frac{n+4}{2}\rceil$.
\end{enumerate}
	 \end{theorem}
	 
\begin{proof}
Since our proofs for both parts are very similar, we give details for 
the first part and only sketch details of the second part.
		 
\noindent
\textbf{Proof of item 1.} By Theorem \ref{spectra-steiner-bistar}, 
it follows that the characteristic polynomial of 
$\DD_2(S_{1,n-3})[B_1,B_1]$ is 
%\begin{align*}
$h(x)=(x+1)^{n-5}\big[x^6-2(n-1)x^5-3(7n-12)x^4-18(3n-7)x^3-5(9n-22)x^2+
(13n-28)x-(n-1)\big].$
%h(x)&=(x+1)^{n-5}\big[x^6-2(n-1)x^5-3(7n-12)x^4-18(3n-7)x^3-5(9n-22)x^2\\
%&~~~~~~~~~~~~~~~~~~~~~~~~-(13n-28)x-(n-1)\big].
			 % &=-\Big[(n-1)+(n^2+7n-23)x+\frac{1}{2}(n^3+14n^2-55n+30)x^2+\frac{1}{6}(n^4+20n^3-118n^2\\
			 % &~~~~~~~~~+91n+234)x^3+\frac{1}{24}(n^5+25n^4-231n^3+299n^2+1562n-3504)x^4\\
			 %&~~~~~~~~~+\frac{1}{120}(n^6+29n^5-410n^4+985n^3+5389n^2-28674n+36480)x^5\Big]\\
			 %&~~~-\sum_{k=6}^{n-5} \bigg[(n-1)\binom{n-5}{k}+(13n-28)\binom{n-5}{k-1}+(45n-110)\binom{n-5}{k-2}+(54n-126)\cdot\\
			 %&~~~~~~~~~~~~~~~\binom{n-5}{k-3}+(21n-36)\binom{n-5}{k-4}+2(n-1)\binom{n-5}{k-5}-\binom{n-5}{k-6}\bigg]x^k\\
			 %&~~~-\Big[\frac{1}{120}(9n^5+185n^4-2755n^3+10255n^2-7334n-14640)x^{n-4}+\frac{1}{24}(7n^4+126n^3\\
			 %&~~~~~~~~-1159n^2+2418n-480)x^{n-3}+\frac{1}{6}(5n^3+72n^2-383n+354)x^{n-2}\\
			 %&~~~~~~~~~+\frac{1}{2}(3n^2+29n-41)x^{n-1}+(n+3)x^n\Big]+x^{n+1}
%		 	\end{align*} 
		 
If $a_k$ is the coefficient of $x^k$ in $h(x)$, then, we have
\begin{align*}
a_0 &=-(n-1),~~~a_1 =-(n^2+7n-23),~~~a_2=-\frac{1}{2}(n^3+14n^2-55n+30),\\
a_3 &=-\frac{1}{6}(n^4+20n^3-118n^2+91n+234),\\
a_4&=-\frac{1}{24}(n^5+25n^4-231n^3+299n^2+1562n-3504)\\
%a_5&=-\frac{1}{120}(n^6+29n^5-410n^4+985n^3+5389n^2-28674n+36480),\\
a_k&=-\sum_{k=6}^{n-5} \bigg[(n-1)\binom{n-5}{k}+(13n-28)\binom{n-5}{k-1}+(45n-110)\binom{n-5}{k-2}+(54n-126)\cdot\\
&~~~\binom{n-5}{k-3}+(21n-36)\binom{n-5}{k-4}+2(n-1)\binom{n-5}{k-5}-\binom{n-5}{k-6}\bigg]~ \text{for}~ 6\leq k\leq n-5,\\
%a_{n-4}&=-\frac{1}{120}(9n^5+185n^4-2755n^3+10255n^2-7334n-14640),\\
a_{n-3}&=-\frac{1}{24}(7n^4+126n^3-1159n^2+2418n-480),\\
a_{n-2}&=-\frac{1}{6}(5n^3+72n^2-383n+354),~a_{n-1}=-\frac{1}{2}(3n^2+29n-41),~ a_n=-(n+3),~ a_{n+1}=1.
		 \end{align*}

It is easy to check when $n \geq 6$ that $|a_0|\leq |a_1|\leq |a_2|$ and 
that $|a_{n-2}|\geq |a_{n-1}|\geq |a_{n}|$.  When $7\leq k\leq n-5$, it is 
again easy to see that we have 
\begin{align*}
&~|a_k|-|a_{k-1}|\\  
&=\binom{n-5}{k-1}\bigg(\frac{(n-1)(n-2k-4)}{k}\bigg)+\binom{n-5}{k-2}\bigg(\frac{(13n-28)(n-2k-2)}{k-1}\bigg)\\
&~~~+\binom{n-5}{k-3}\bigg(\frac{(45n-110)(n-2k)}{k-2}\bigg)+\binom{n-5}{k-4}\bigg(\frac{(54n-126)(n-2k+2)}{k-3}\bigg)\\
&~~~+\binom{n-5}{k-5}\bigg(\frac{(21n-36)(n-2k+4)}{k-4}\bigg)+\binom{n-5}{k-6}\bigg(\frac{2(n-1)(n-2k+6)}{k-5}\bigg)\\
&~~~-\binom{n-5}{k-7}\bigg(\frac{(n-2k+8)}{k-6}\bigg).
\end{align*}
Hence, when $7\leq k\leq n-5$, one can check that
$|a_k|\geq |a_{k-1}|$ if and only if $k\leq \frac{n-4}{2}$ and 
$|a_k|\leq |a_{k-1}|$ if and only if $k\geq \frac{n+6}{2}$. Thus, we have  
$|a_0|\leq |a_1|\leq |a_2| \leq \hdots \leq |a_{\lfloor \frac{n-4}{2} 
\rfloor}|$ 
and $|a_{\lceil \frac{n+6}{2}-1\rceil}| \geq |a_{\lceil \frac{n+6}{2}\rceil}| 
\geq \hdots \geq |a_{n-2}|\geq |a_{n-1}|\geq |a_{n}|$. Hence, if 
$|a_t|=\max_{0\leq k\leq n-2} |a_k|$, then $\lfloor \frac{n-4}{2} \rfloor 
\leq t \leq \lceil \frac{n+4}{2}\rceil$. This completes the proof of
the first part.
		 
\noindent \textbf{Proof of item 2.} By Theorem \ref{spectra-steiner-bistar}, 
it follows that the characteristic polynomial of $\DD_2(S_{1,n-3})[B_2,B_2]$ is 
$h(x)=(x+1)^{n-5}\big[x^6-2(n-1)x^5-(21n-22)x^4-(62n-141)x^3-(53n-133)x^2-
(15n-34)x-(n-1)\big]$. As the rest of the proof is similar to the first case, 
we omit its details. 
\end{proof}

\subsection{Bounds on the peak location of the Path}
\label{subsec:peak-locn-path}
	 
\def\DPn{\mathfrak{D}_{P_{n}}}
	 \def\DP#1{\mathfrak{D}_{P_{#1}}}
	 \def\dST{\mathrm{d}_{\mathrm{ST}}}
	 \def\1{\mathbf{1}}
	 \def\0{\mathbf{0}}
	 \def\v{\mathbf{v}}
	 
	 %\blue{Change $\mathfrak{B}$ to $B$ in this subsection. }
	 
For a matrix $M$ and index sets $\alpha$ and $\beta$, the submatrix of $M$ restricted to the rows 
in  $\alpha$  and the columns in  $\beta$ is denoted by 
$M[\alpha,\beta]$.
When $\alpha=\beta$, we use the notation $M[\alpha]$ to denote the principal
submatrix $M[\alpha,\alpha]$ of $M$. 
We also use the notation $M(\alpha|\beta)$ to denote the submatrix of $M$ 
obtained by deleting the rows corresponding to $\alpha$  and the 
columns corresponding to $\beta$. We recall some results 
from \cite{aliazimi-siva-steiner-2-dist}.

\begin{lemma} \cite[Lemma 9, 11]{aliazimi-siva-steiner-2-dist} 
\label{lem:steiner-v}
Suppose $T$ is a  a tree of order $n$ with $p$ leaves. 
Let $B$ be a basis of $\DD_2(T)$ as defined in Remark 
\ref{remark: basis-2-steiner} with 
$B =\{e_1,e_2,\ldots,e_{n-1},f_1,\ldots,f_{n-p}\}$ 
and let $\v$ be the column vector defined as
$v_{e_i} = 1- \sum_{e_i\in f_j'} \left(|B_j|-1\right)$ and 
$v_{f_i} = |B_i|-1, \text{ where $f_i'\in B_i$}.$
Then $\1^t\v =1$ and $\DD_2[B,B]\v=(n-1)\1$. 
\end{lemma}
	 
	 \begin{remark}
		 \label{remark:v-for-path}
		 When $T=P_n$, a path on $n$ vertices, the vector $\v$ 
		 defined in Lemma \ref{lem:steiner-v} is given by 
		 $
		 v_{f_j} =1$ for $1\leq j\leq n-p$, and for $1\leq i\leq n-1$, $v_{e_i}=\begin{cases}
			 0 & \text{if $e_i$ is a pendant edge},\\
			 -1 & \text{otherwise;}
		 \end{cases}
		 $
	 \end{remark}

\begin{remark}
\label{remark: basis Pn}
Let $P_n$ be the path on $n$ vertices with edges $e_i=\{i,i+1\}$ 
for $i=1,\ldots, n-1$. Let  $B=(e_1,f_1,e_2,f_2,\ldots,
e_{n-2},f_{n-2},e_{n-1})$ be the ordered basis for the 
row space of $\DD_2(P_n)$ where  $f_j=\{j,j+2\}$,  
for $j=1,\ldots, n-2$.  We follow this particular ordering 
of $B$ to order the rows and columns of $\DD_2(P_n)[B,B]$.
\end{remark}
	 
We denote by $\DPn$ the matrix 
$\DD_2(P_n)[B, B]$, that is, $\DPn:=\DD_2(P_n)[B, B]$.

\begin{remark}\label{remark:laplacian-like}
By the definition of the Laplacian-type matrix outlined 
in \cite[Page 77]{aliazimi-siva-steiner-2-dist}, we define the
matrix $L$ whose rows and columns are indexed by the elements of 
$B$ with entries as follows:
the entries $L(e_i,e_j)$ and $L(f_i,f_j)$ are zero if $i\neq j$. 
Further, define
$$
L(x,x) =\begin{cases}
2 & \text{if  $x\in B\setminus\{e_1,e_{n-1}\}$},\\
1 & \text{if $x\in \{e_1,e_{n-1}\}$}, 
\end{cases}, \quad \text{and} \quad L(e_i,f_k) = \begin{cases}
-1 & \text{if $i\in\{k-1,k\}$},\\
0 & \text{otherwise}.
\end{cases}$$
Note that $L$ is a symmetric tridiagonal matrix 
$$
L=\left(\begin{array}{rrrrrrr}
1  & -1 & 0  & 0  & \cdots  & 0    \\
-1 & 2  & -1 & 0  & \cdots  & 0    \\
0  & -1 & 2  & -1 & \cdots  & 0    \\
\vdots  & \vdots   &\ddots  & \ddots  &\ddots   & \vdots    \\
0  & 0  & \cdots & -1  & 2  & -1 \\
0  & 0  &    \cdots  & 0   & -1 & 1 
\end{array}\right).
$$
\end{remark}
	 
The following result is a special case of  
\cite[Theorem 1 and  2]{aliazimi-siva-steiner-2-dist} and provides 
the inverse of  $\DPn$.  
	 
\begin{theorem}\label{th:det-inv-st-path}
Let $P_n$ is the path on $n$ vertices  and  
let $B=(e_1,f_1,e_2,f_2,\ldots,e_{n-2},f_{n-2},e_{n-1})$ 
be the ordered basis  of $\DD_2(P_n)$ as defined in Remark 
\ref{remark: basis Pn}.  Then  $\det\DPn = (n-1)$ and 
$ \DPn^{-1} = -L +\frac{1}{n-1} \v \v^t.$
\end{theorem}
	 
We need the following result (see  Horn and Johnson \cite[Page 18]
{horn-johnson-matrix-analysis}), about the blocks in the inverse of a 
	 partitioned nonsingular matrix $M$.
	 \begin{lemma}\label{lem:partition-inverse}
		 Let $M$ be a nonsingular matrix and $\alpha$ be a subset of the 
		 index set of $M$'s rows and columns.  Let $\alpha^c$ denote 
		 the complement set of $\alpha$ and suppose $M^{-1}[\alpha]$ 
		 and $M[\alpha^c]$ are invertible. Then,
		 $$\left(M^{-1}[\alpha]\right)^{-1}= M[\alpha]-M\left[\alpha, \alpha^c\right]M\left[\alpha^c\right]^{-1} M\left[\alpha^c, \alpha\right].$$
	 \end{lemma}

In our next result, we find the principal minors of $\DPn$ of size $r-1$. 
	 
\begin{theorem}\label{th:principal-monors}
Let $P_n$ be a path on $n\geq 3$ vertices and $\DD_2(P_n)$ be 
its $2$-Steiner distance matrix. If $B$ is a basis of $\DD_2(P_n)$'s row space and $\alpha \in B$, then
$$
\det \DPn(\alpha | \alpha) = 
		 \begin{cases}
			 -(n-1) & \text{ if $\alpha$ is a pendant edge in $P_n$},\\
			 -(2n-3) & \text{otherwise}.
		 \end{cases}
		 $$
	 \end{theorem}
	 
	 \begin{proof} 
		 Our proof is by induction on $n$. Let $B=\{e_1,f_1,e_2\}$ be a 
		 basis of $\DD_2(P_3)$. Note that $\DP3=\left(\begin{array}{rrr}
			 1 & 2 & 2 \\
			 2 & 2 & 2 \\
			 2 & 2 & 1
		 \end{array}\right)$. Clearly $\DP3(f_1 | f_1)=-3$ and  
		 $\DP3(e_i | e_i)=-2$ when $i=1,2$. Hence, the result holds 
		 when $n=3$. Further, note that 
		 $\DP4 = \left(\begin{array}{rrrrr}
			 1 & 2 & 2 & 3 & 3 \\
			 2 & 2 & 2 & 3 & 3 \\
			 2 & 2 & 1 & 2 & 2 \\
			 3 & 3 & 2 & 2 & 2 \\
			 3 & 3 & 2 & 2 & 1
		 \end{array}\right).$ 
		 One can verify that $\det \DP4(e_1|e_1) =\det \DP4(e_3|e_3)=-3$ 
		 and that $\det\DP4(\alpha|\alpha)=-5$ for $\alpha\in \{f_1,e_2,f_2\}$. 
		 Hence, our result is also true when $n=4$.

		 Assume that the statement is true for all path on $k$ vertices, where $k\leq n-1$. 
		 By $P_n$ we mean the  path on $n>4$ vertices with  $e_i=\{i,i+1\}$ 
		 for $i=1,\ldots, n-1$ and let $f_j=\{j,j+2\}$ for $j=1,\ldots, n-2$. 
		 Let $B_n=(e_1,f_1,e_2,f_2,\ldots,e_{n-2},f_{n-2},e_{n-1})$ be an 
		 ordered basis of $\DD_2(P_n)$'s row space and 
		 $\DPn= \DD_2(P_n)[B_n,B_n]$. 
		 Further note that $\dST(e_1,b_i) =\dST(f_1,b_i)$ for each 
		 $b_i\in B_n\setminus \{e_1\}$ and  $\dST(f_1,b_i) =\dST(e_2,b_i)$ 
		 for each $b_i\in B_n\setminus \{e_1,f_1\}$.  
		 For $x\in B_n$, we write $r_x$ (respectively $c_x$) to denote 
		 the row (respectively column)  corresponding to $x$. By 
		 performing the elementary row operations 
		 $r_{f_1}=r_{f_1}-r_{e_2}$ and $c_{f_1}=c_{f_1}-c_{e_2}$ 
		 on the matrix $\DPn(e_1|e_1)$ we get 
		 $
		 \left( \begin{array}{r|c}
			 -1 & \1^t\\ \hline 
			 \1 & \DP{n-1}
		 \end{array}\right),
		 $
		 where $\DP{n-1}=\DD_2(P_{n-1})[B_{n-1},B_{n-1}]$.  By 
		 Lemma \ref{lem:steiner-v}, there exist $v$ such that 
		 $\1^tv=1$ and $\DP{n-1}v=(n-2)\1$.   By Schur complements and the 
		 determinantal formula \cite[sec. 0.8.5]{horn-johnson-matrix-analysis} 
		 and Theorem \ref{th:det-inv-st-path}, we get 
		 \begin{align*}
			 \det \DPn(e_1|e_1) = \det(\DP{n-1}) \left(-1- \1^t \DP{n-1}^{-1}\1\right) = (n-2) \left(-1-  \frac{\1^t v}{n-2}\right) =-(n-1).
		 \end{align*}
		 
		 Analogously, we have $ \det \DPn(e_{n-1}|e_{n-1}) =-(n-1)$.  
		 Let $\alpha\in \{f_1,e_{2},\ldots, e_{n-2}, f_{n-2}\}$. Since $n>4$, 
		 without loss of generality, we may assume that $\{e_1,f_1,e_2\} 
		 \subset \alpha^c$. By performing the row operations 
		 $r_{f_1}=r_{f_1}-r_{e_2}$ and $c_{f_1}=c_{f_1}-c_{e_2}$ on 
		 the matrix $\DPn(\alpha|\alpha)$ we get 
		 \begin{align*}
			 \DPn(\alpha|\alpha) \sim   \left(\begin{array}{rr|c}
				 -2 & 1 & \1^t  \\
				 1 & -1  & \0^t \\  \hline 
				 \1 & \0 & \DP{n-1}(\alpha|\alpha) 
			 \end{array}\right).
		 \end{align*}
		 
		 Again, by applying Schur complements and the determinantal formula, 
		 we get 
		 \begin{equation}
			 \det \DPn(\alpha|\alpha)  =\det  \DP{n-1}(\alpha|\alpha)  \det \left[
			 \left(\begin{array}{cc}
				 -2 & 1 \\
				 1 & -1   
			 \end{array}\right)
			 - X^t  \DP{n-1}(\alpha|\alpha)^{-1} X
			 \right], \label{eq:schur-complement-minor}
		 \end{equation}
		 where $X= \begin{pmatrix}	\1 & \0 \end{pmatrix} $.  By Lemma \ref{lem:partition-inverse}, we get 
		 \begin{equation}\label{eq:inverse-D-partition}
			 \DP{n-1}[\alpha^c]^{-1} =  \DP{n-1}^{-1}[\alpha^c] - \DP{n-1}^{-1}[\alpha^c,\alpha] \left(\DP{n-1}^{-1}[\alpha]\right)^{-1}\DP{n-1}^{-1}[\alpha,\alpha^c].
		 \end{equation}
		 Let $L$ be the Laplacian-like matrix for the tree $P_{n-1}$, as described in Remark \ref{remark:laplacian-like}. Suppose $v[\alpha] =t$. Clearly $t\in\{-1,1\}$. By Theorem \ref{th:det-inv-st-path}, we get 
		 \begin{equation}\label{eq:m11}
			 \DP{n-1}^{-1}[\alpha] = -L[\alpha] + \frac{1}{n-2} v[\alpha](v[\alpha])^t = -2+\frac{1}{n-2} =-  \frac{2n-5}{n-2}.
		 \end{equation}
		 Note that $\1^tv[\alpha^c]\1 +v[\alpha] =1$. Since $L\1=\0$, it follows that $\1^tL[\alpha^c]\1  = 2$. Hence, by Theorem \ref{th:det-inv-st-path}, we get 
		 \begin{equation} \label{eq:m12}
			 X^t  \DP{n-1}^{-1}[\alpha^c] X =  - \begin{pmatrix}
				 \1^t L[\alpha^c]\1 & 0 \\
				 0 & 0
			 \end{pmatrix} +
			 \frac{1}{n-2}
			 \begin{pmatrix}
				 \1^t v[\alpha^c] (v[\alpha^c])^t\1 & 0 \\
				 0 & 0
			 \end{pmatrix}  = \begin{pmatrix}
				 -2 + \dfrac{(1-t)^2}{n-2} & 0\\0&0
			 \end{pmatrix}.
		 \end{equation}
		 Further, note that 
		 \begin{equation} \label{eq:m13}
			 X^t  \DP{n-1}^{-1}[\alpha^c,\alpha]  = X^t
			 \left[ -L[\alpha^c,\alpha] + \frac{1}{n-2} v[\alpha^c](v[\alpha])^t\right]
			 = \begin{pmatrix}
				 2 + \dfrac{t(1-t)}{n-2} \\0
			 \end{pmatrix}.
		 \end{equation}
		 \begin{equation} \label{eq:m14}
			 \DP{n-1}^{-1}[\alpha,\alpha^c] X =
			 \left[ -L[\alpha] + \frac{1}{n-2} v[\alpha](v[\alpha])^t\right]X
			 = \begin{pmatrix}
				 2 + \dfrac{t(1-t)}{n-2}  & 0
			 \end{pmatrix}.
		 \end{equation}
		 By \eqref{eq:inverse-D-partition}, \eqref{eq:m11}, \eqref{eq:m12}, \eqref{eq:m13}, and \eqref{eq:m14}, we get 
		 \begin{equation*}  
			 X^t  \DP{n-1}(\alpha|\alpha)^{-1} X = \begin{pmatrix}
				 -2 + \dfrac{(1-t)^2}{n-2} & 0\\0&0
			 \end{pmatrix} +\frac{n-2}{2n-5} \begin{pmatrix}
				 \left[2 + \dfrac{t(1-t)}{n-2} \right]^2 & 0\\0&0
			 \end{pmatrix} .
		 \end{equation*}
		 On simplification we get
		 \begin{equation} \label{eq:m15}
			 X^t  \DP{n-1}(\alpha|\alpha)^{-1} X =  \begin{pmatrix}
				 \frac{2}{2n-5}+f(t)& 0\\0&0
			 \end{pmatrix},
		 \end{equation}
		 where   $f(t) = \frac{(1-t)^2}{n-2}+\frac{4t(1-t)}{2n-5}+\frac{t^2(1-t)^2}{(n-2)(2n-5)}
		 $. It is easy to note that $f(\pm 1)=0$. Hence, it follows from \eqref{eq:m15} that
		 \begin{equation} \label{eq:m16}
			 X^t  \DP{n-1}(\alpha|\alpha)^{-1} X  = \begin{pmatrix}
				 \frac{2}{2n-5}& 0\\0&0
			 \end{pmatrix}.
		 \end{equation}
		 By \eqref{eq:schur-complement-minor} and \eqref{eq:m16} we get 
		 $$
		 \det \DPn(\alpha|\alpha)  =  \frac{2n-3}{2n-5}\ \det \DP{n-1}(\alpha|\alpha)  .
		 $$
		 Thus, the result follows by induction and our proof
		 is complete.
	 \end{proof}

	 %%%%%%%%%%%%%%%%%%%%%%%%%%%%%%%%%%%%%%%%%%%%%%%%%%%%%%%%%%%%%%%%%%%%%%%%%%%%%%

	 We next present our upper bound on the peak location for coefficients
	 in the characteristic polynomial of 
	 $\DPn=\DD_2(P_n)[B,B]$, where the basis $B$ is defined as in Remark \ref{remark: basis Pn}.   .

	 \begin{theorem}
		 \label{th: peak-steiner-path}
		 Let $P_n$ be a path on $n> 2$ vertices and $\charpoly_{\DPn}(x) = 
		 \sum_{i=0}^{2n-3}a_ix^i$.  If $|a_\ell| = \max\{|a_0|,|a_1|, \ldots, |a_{2n-4}|\}$, then $\ell \leq \left\lfloor\dfrac{7n}{5}\right\rfloor$.
	 \end{theorem}
	 
	 \begin{proof}
To prove the result we will use Lemma 3.2(1) of\cite{Abiad-Brimkov-Hayat-Khramova-Koolen-dist-char-poly}. Since $\det (\DPn) = (n-1)$, we have $|a_0| = n-1$. Furthermore, by Theorem \ref{th:principal-monors}, the sum of all principal minors of $\det \DD_2(P_n)[B,B]$ of size $2n-4$ is given by
		 $$
		 -(2n-5)(2n-3) -2(n-1) = -(4n^2-14n+13).
		 $$
		 It follows that $|a_1|= 4n^2-14n+13$.  Now note that 
		 $$
		 \dfrac{(2n-3)-j}{(2n-3)(j+1)} \dfrac{4n^2-14n+13}{n-1} <1  \iff  j > \dfrac{(2n-3)(4n^2-15n+14)}{3(2n-3)(n-2)+2(n-1)} =f(n)
		 $$
		 Suppose $g(n) = \frac{7n}{5}$. Note that $g'(n) - f'(n) > 0$ for $n > 2$. Hence, by \cite[Lemma 3.2(1)]{Abiad-Brimkov-Hayat-Khramova-Koolen-dist-char-poly}, it follows that $\ell \leq  \left\lfloor\dfrac{7n}{5}\right\rfloor$.
	 \end{proof}

	 Note that Theorem \ref{th: peak-steiner-path} only provides an 
	 upper bound for the peak location of the unimodal sequence 
	 ${|a_0|,\ldots, |a_{2n-4}|}$ associated to a path $P_n$. One can 
	 use the approach mentioned in 
	 \cite[Lemma 3.2(2)]{Abiad-Brimkov-Hayat-Khramova-Koolen-dist-char-poly} 
	 to get a lower bound on the peak location.  However, to use 
	 \cite[Lemma 3.2(2)]{Abiad-Brimkov-Hayat-Khramova-Koolen-dist-char-poly}, 
	 a suitable estimate of $a_{2n-4}$ and $a_{2n-5}$ is required. 
	 In the case of $P_n$, even if $a_{2n-4}$ and $a_{2n-5}$ are 
	 known exactly, 
	 \cite[Lemma 3.2(2)]{Abiad-Brimkov-Hayat-Khramova-Koolen-dist-char-poly} 
	 does not seem to provide a lower bound on the peak location, and 
	 so we do not discuss this aspect in this paper.  Using 
	 SageMath \cite{sage},  when  $5<n<15$, the actual peak location 
	 for $P_n$ seems to be $n-1$.  We record this as a conjecture.

	 \begin{conjecture}
		 For a path $P_n$ on $n> 5$ vertices, if $\charpoly_{\DPn}(x) = \sum_{i=0}^{2n-3}a_ix^i$ and $|a_\ell| = \max\{|a_0|,|a_1|, \ldots, |a_{2n-4}|\}$, then $\ell  =n-1$.
	 \end{conjecture}
	 
	 We further note that $|a_{2n-4}|$ is the trace of $\DPn$, and 
	 hence $|a_{2n-4}|=2(n-2) +(n-1) =3n-5$.  One needs to find principal 
	 minors of a suitable size to estimate $a_{2n-5}$. Again, by looking
	 at the data from SageMath, we make the following conjecture 
	 that provides an estimate for $a_{2n-5}$.

	 \begin{conjecture}
		 For a path $P_n$ on $n> 5$ vertices, if $\charpoly_{\DPn}(x) = \sum_{i=0}^{2n-3}a_ix^i$, then 
		 $a_{2n-5}  =- \frac{1}{6} (n - 1) (n - 2) (2n^2 + 6n - 15).$
	 \end{conjecture}

	%section 2

		%%%%%%%%%%%%%%%%%%%%%%%%
		% Bibliography
		%%%%%%%%%%%%%%%%%%%%%%%%
\bibliographystyle{abbrv}		 
		 \newcommand{\etalchar}[1]{$^{#1}$}

	%\section*{References}
	
	%\bibliographystyle{alpha}
	%%\bibliographystyle{apalike}%
%	\bibliography{bibdatabase}

\end{document}